\documentclass{amsart}
\usepackage[utf8]{inputenc}
\usepackage{amsmath}
\usepackage{amsthm}
\usepackage{amsfonts, dsfont}
\usepackage{amssymb}
\usepackage[all]{xy}
\usepackage{indentfirst}
\usepackage[pagebackref=true]{hyperref}

\usepackage{cleveref}
\usepackage[alphabetic, initials]{amsrefs}

\usepackage{faktor}
\usepackage{xfrac}

\usepackage{mathbbol}
\usepackage{amssymb}             

\DeclareSymbolFontAlphabet{\amsmathbb}{AMSb}

\newtheorem{thm}{Theorem}[section]
\newtheorem{prop}[thm]{Proposition}

\newtheorem{theorem}[thm]{Theorem}
\newtheorem{corollary}[thm]{Corollary}
\newtheorem{proposition}[thm]{Proposition}
\newtheorem{question}[thm]{Question}
\newtheorem{lemma}[thm]{Lemma}
\newtheorem*{theorem*}{Theorem}

\theoremstyle{definition}

\newtheorem{defn}[thm]{Definiton}
\newtheorem*{defn*}{Definiton}
\newtheorem{example}[thm]{Example}
\newtheorem{definition}[thm]{Definition}

\newtheorem{remark}[thm]{Remark}
\newtheorem*{remark*}{Remark}

\newcommand{\N}{\mathbb{N}} 
\newcommand{\Z}{\mathbb{Z}} 
\newcommand{\bC}{\mathbb{C}} 
\newcommand{\R}{\mathbb{R}}

\usepackage{verbatim} 
\usepackage{color}
\newcommand{\G}{\Gamma}

\DeclareMathOperator{\interior}{int}

\newcommand{\La}{\Lambda}
\newcommand{\la}{\lambda}

\newcommand{\Ind}{\operatorname{Ind}}
\newcommand{\Sub}{\operatorname{Sub}}
\newcommand{\Stab}{\operatorname{Stab}}
\newcommand{\Rep}{\operatorname{Rep}}

\newcommand{\act}{\!\curvearrowright\!}
\newcommand{\pr}{\operatorname{Prob}}

\newcommand{\supp}{\operatorname{supp}}

\newcommand{\cB}{\mathcal{B}}
\newcommand{\cC}{\mathcal{C}}

\newcommand{\cH}{\mathcal{H}}

\newcommand{\cO}{\mathcal{O}}
\newcommand{\cP}{\mathcal{P}}

\newcommand{\cU}{\mathcal{U}}
\newcommand{\cX}{\mathcal{X}}

\newcommand{\cW}{\mathcal{W}}

\renewcommand{\k}{\kappa}
\newcommand{\one}{\boldsymbol{1}}

\newcommand{\Rad}{\mathrm{Rad}}
\renewcommand{\b}{\mathfrak{b}}

\newcommand{\csx}{C^*_{\la_{\sfrac{\G}{\G_x}}}\!\!(\G)}

\newcommand{\csxo}{C^*_{\la_{\sfrac{\G}{\G_x^{\scriptscriptstyle\tiny 0}}}}\!(\G)}

\newcommand{\Xoo}{X_{\hspace{-0.05em}\scriptscriptstyle 0}^{\hspace{-0.05em}\scriptscriptstyle 0}}

\title{Boundary maps, germs and quasi-regular representations}
\author{Mehrdad Kalantar}
\address{Mehrdad Kalantar\\ University of Houston\\ USA}
\email{mkalantar@uh.edu}
\author{Eduardo Scarparo}
\address{Eduardo Scarparo \\ Department of Mathematical Sciences, NTNU, NO-7491 Trondheim, Norway}
\email{eduardo.scarparo@ntnu.no}

\thanks{MK is supported by a Simons Foundation Collaboration Grant (\# 713667). This work was carried out during the tenure of an ERCIM ‘Alain Bensoussan’ Fellowship Programme.}

\begin{document}

\maketitle
\begin{abstract}
We investigate the tracial and ideal structures of $C^*$-algebras of quasi-regular representations of stabilizers of boundary actions.
\\
Our main tool is the notion of boundary maps, namely $\G$-equivariant unital completely positive maps from $\G$-$C^*$-algebras to $C(\partial_F\G)$, where $\partial_F\G$ denotes the Furstenberg boundary of a group $\G$.
\\
For a unitary representation $\pi$ coming from the groupoid of germs of a boundary action, we show that there is a unique boundary map on $C^*_\pi(\G)$. Consequently, we completely describe the tracial structure of the $C^*$-algebras $C^*_\pi(\G)$, and for any $\G$-boundary $X$, we completely characterize the simplicity of the $C^*$-algebras generated by the quasi-regular representations $\la_{\G/\G_x}$ associated to stabilizer subgroups $\G_x$ for any $x\in X$.
\\
As an application, we show that the $C^*$-algebra generated by the quasi-regular representation $\la_{T/F}$ associated to Thompson's groups $F\leq T$ does not admit traces and is simple.
\end{abstract}


\section{Introduction}
Furstenberg's theory of topological boundaries has recently found striking applications in $C^*$-algebras associated to groups and group actions, starting with the characterization of $C^*$-simplicity in terms of freeness of the Furstenberg boundary action in \cite{KalKen}, and then the characterization of the unique trace property in terms of faithfulness of the Furstenberg boundary action in \cite{BKKO}. Other characterizations of $C^*$-simplicity were proved in series of subsequent work by various authors (\cite{Haag15, Ken15, HartKal, Raum}).

One of the primary goals of this paper is to unify the above approaches and techniques in order to generalize the applications of boundary actions to the case of $C^*$-algebras generated by certain quasi-regular representations. 

Several efforts have also been made towards extending the above applications in similar problems beyond reduced group $C^*$-algebras and crossed products (\cite{Kaw, KenBryd, KenSchaf, BekKal, Nagh, BearKal, Mon19, Borys}), where also various generalizations of the notion of boundary actions have been proposed and studied.

Although these results appear different, the dynamics of the Furstenberg boundary action is present (explicitly or implicitly) in the background of all the notions and proofs involved.

\medskip

The main tool in this paper is the notion of boundary maps:
\begin{defn*}
Let $A$ be a $\G$-$C^*$-algebra. A \emph{boundary map} on $A$ is a $\G$-equivariant unital completely positive map $\psi \colon A\to C(\partial_F\G)$.
\end{defn*}

The importance of boundary maps was made clear by Kennedy in \cite{Ken15}, who showed that a group $\G$ is $C^*$-simple if and only if the only boundary map on its reduced $C^*$-algebra $C^*_{\la_\G}(G)$ is the canonical trace.

Our results show that this notion is a fundamental concept in more general contexts. Note that, thanks to $\G$-injectivity of $C(\partial_F\G)$ \cite[Theorem 3.11]{KalKen}, boundary maps always exist.
Thus, in the case of inner actions, uniqueness and faithfulness of boundary maps entail strong implications for the ideal and tracial structures of the $C^*$-algebras in question.

In this paper, we are particularly interested in $C^*$-algebras $\csx$ generated by quasi-regular representations of stabilizer subgroups $\G_x$ of boundary actions $\G\act X$, as well as the $C^*$-algebras $C^*_{\pi\times\rho}(\G,X)$ generated by certain covariant representations of the action $\G\act X$, called \emph{germinal representation of $(\G,X)$} (see Definition~\ref{grprep}).
The class of germinal representations includes all Koopman representations of quasi-invariant measures on $X$, as well as the canonical covariant representation of $(\G, X)$ in $B(\ell^2(\G/H))$ for every $\G_x^0\leq H\leq \G_x$.

\medskip

If $X$ is a $\G$-boundary, we let $\b_X\colon\partial_F\G\to X$ be the unique continuous $\G$-equivariant map. Given $g\in\G$, we let $X^g:=\{x\in X:gx=x\}$, and we use the shorthand notation $\Delta_g$ for the set $\overline{\b_{\hspace{-0.1em} X}^{-1}(\interior X^g)}$. The set $\Delta_g$ turns out to be central in our results.

We derive several important implications of the following uniqueness result for boundary maps.

\begin{theorem*}[Theorem~\ref{thm:uniquemap-koop}]
Let $X$ be a $\G$-boundary and $(\pi,\rho)$ a germinal representation of $(\G,X)$. Then there is a unique boundary map $\psi$ on $C^*_{\pi\times\rho}(\G,X)$, and $\psi|_{C^*_\pi(\G)}$ is the unique boundary map on $C^*_\pi(\G)$.

Furthermore, $\psi({\pi}(g))=\one_{\Delta_g}$ and ${\psi}(\rho(f))=f\circ\b_{\hspace{-0.1em} X}$  for all $g\in \G$ and $f\in C(X)$.
\end{theorem*}

Consequently, for germinal representations $\pi$, we completely describe the tracial structure of $C^*_\pi(\G)$.
\begin{theorem*}[Theorem \ref{faithful-trace}]
Let $X$ be a faithful $\G$-boundary, and $\pi$ a germinal representation of $(\G,X)$. Then $C^*_{\pi}(\G)$ admits a trace if and only if $X$ is topologically free. 

If $X$ is topologically free, then ${\pi}\succ\la_\G$ and the canonical trace is the unique trace on $C^*_{\pi}(\G)$.
\end{theorem*}

The problem of simplicity is more subtle. Nevertheless we are still able to completely characterize $C^*$-simplicity of the quasi-regular representations $\la_{\G/\G_x}$ associated to stabilizer subgroups of any point $x\in X$ for any $\G$-boundary $X$.

\begin{theorem*}[Theorem~\ref{thm:C*-simple-main}]
Let $X$ be a $\G$-boundary, and $x\in X$. Then $\csx$ is simple iff the quotient group {\small$\displaystyle \frac{\G_x}{\G_x^0}$} is amenable. 
\end{theorem*}

Given a compact $\G$-space $X$, we denote by $\Xoo$ the set of continuity points of the open stabilizer map $\Stab^0\colon X\to\Sub(\G)$, $x\mapsto \G_x^0$. 

For general germinal representations $\pi$ coming from boundary actions, we conclude the following.
\begin{theorem*}[Theorem~\ref{thm:prec} and Corollary~\ref{cor:xoo}]
Let $X$ be a $\G$-boundary and $\pi$ a germinal representation of $\G$. Given $x\in \Xoo$ and $\sigma$ a unitary representation of $\G$ such that $\sigma\prec\pi$, we have that $\la_{\G/\G_x^0}\prec\sigma$; consequently, $\csxo$ is simple for every $x\in \Xoo$. 
\end{theorem*}


Given a group $\G$, $C^*_{\la_\G}(\G)$ always admits a trace. Furthermore, $\G$ is amenable if and only if $C^*_{\la_\G}(\G)$ is nuclear and admits a one-dimensional representation. In particular, unless $\G$ is trivial, nuclearity and simplicity of $C^*_{\la_\G}(\G)$ are properties far apart from each other. 

For quasi-regular representations, however, the situation is quite different. In \cite{HaagOle} Haagerup and Olesen showed that there is a quasi-regular representation $\pi$ of Thompson's group $V$ such that $C^*_\pi(V)$ is isomorphic to the Cuntz algebra $\cO_2$ (in particular, it is nuclear, simple, and admits no traces). In \cite{BS}, Brix and the second named author generalized this result for topological full groups of ample groupoids.

Let $T$ and $F$ be Thompson's groups. Haagerup and Olesen also observed that $C^*_\pi(T)\subsetneq C^*_\pi(V)$, and that $\pi|_T$ is unitarily equivalent to the quasi-regular representation $\la_{T/F}$. However, the problem of analyzing the structure of $C^*_{\la_{T/F}}(T)$ was left untouched. 

As an application of our results, we prove the following.
\begin{theorem*}[Theorem~\ref{applFT}]
The $C^*$-algebra $C^*_{\la_{T/F}}(T)$ admits no traces and is simple.
\end{theorem*}


In addition to the introduction, this paper has six other sections. In Section \ref{sec:prelim} we recall some definitions and basic facts, and fix the notation that we will be using in the rest of the paper.

In Section \ref{sec:gen}, we begin our study of boundary maps. We gather some of the key facts and techniques, which are known to experts and have been used in different forms. The main purpose of this section is to formulate an abstract framework in which these techniques can be used in order to provide more clarity on the existing ideas which we improve upon in the proofs of our main results.

In Section \ref{section: qrb} we prove  
uniqueness of boundary maps on $C^*$-algebras of germinal representations of boundary actions (Theorem~\ref{thm:uniquemap-koop}). We use this result to describe traces on these $C^*$-algebras.

In Section \ref{sec:simpl}, we turn our attention to $C^*$-simplicity. We apply Theorem \ref{thm:uniquemap-koop} to conclude results on the ideal structure, and simplicity of the $C^*$-algebras of quasi-regular representations of stabilizers of boundary actions.

More concretely, in Section \ref{sec:thomp}, we apply our results from previous sections to analyze certain quasi-regular representations of Thompson's groups.

Section \ref{section:groupoid} is a short overview of some concepts from groupoid theory.
We remark that despite the terminology used, the proofs of the main results of this work do not involve or require groupoid theory. Nonetheless, some of the ideas used in the proofs of our results in Sections \ref{section: qrb} and \ref{sec:simpl} come from groupoid theory. The main purpose of Section \ref{section:groupoid} is to explain this connection. 

\medskip

Finally, let us remark that in the recent preprint~\cite{KS21}, the authors have extended some of the ideas and techniques used in this work to the study of the ideal structure of $C^*$-algebras generated by germinal representations of more general group actions. For instance, among other results, we generalize  Corollary~\ref{cor:cp} to the setting of minimal actions of countable groups on locally compact spaces~(\cite[Corollary 5.2]{KS21}).  \\

\noindent
\textbf{Acknowledgements.}
We are grateful to the anonymous referee for their insightful comments and suggestions which, in particular, resulted in substantial improvement of the presentation of the paper.

\section{Preliminaries}\label{sec:prelim}

\subsection{Group actions}

Throughout the paper $\G$ is a discrete group, and $\G\act X$ denotes an action of $\G$ by homeomorphisms on a compact Hausdorff space $X$. In this case we say $X$ is a compact $\G$-space. The action $\G\act X$ induces an adjoint action of $\G$ on the weak*-compact convex set $\pr(X)$ of regular probability measures on $X$. Given $f\in C(X)$, let $\supp f:=\overline{\{x\in X:f(x)\neq 0\}}$.

Given a $\G$-space $X$ and $x\in X$, the \emph{stabilizer subgroup} of $x$ is $\G_x:=\{g\in\G : gx=x\}$; and the \emph{open stabilizer} is the subgroup $$\G_x^0:=\{g\in \G : \text{ $g$ fixes an open neighborhood of $x$}\}.$$ Observe that $\G_x^0$ is a normal subgroup of $\G_x$ for any $x\in X$. 

We denote by $X^g$ the fixed set of $g\in\G$, that is, the set $X^g:=\{x\in X : gx=x\}$. Note that $g\in \G_x$ iff $x\in X^g$, and $g\in \G_x^0$ iff $x\in \interior X^g$.

We recall the standard terminology for a given action $\G\act X$. We say
\begin{itemize}
\item
the action is \emph{minimal}, or $X$ is a minimal $\G$-space, if $X$ has no non-empty proper closed $\G$-invariant subsets;
\item
the action is \emph{free}, if $X^g =\emptyset$ for every non-trivial $g\in \G$;
\item
the action is \emph{topologically free}, or $X$ is a topologically free $\G$-space, if $\interior X^g =\emptyset$ for every non-trivial $g\in \G$;
\item
the action is \emph{faithful}, or $X$ is a faithful $\G$-space, if $X^g \neq X$ for every non-trivial $g\in \G$.
\end{itemize}

We denote by $\Sub(\G)$ the set of subgroups of $\G$, endowed with the Chabauty topology; this is the restriction to $\Sub(\G)$ of the product topology on $\{0, 1\}^\G$, where 
every subgroup $\La\in \Sub(G)$ is identified with its characteristic function $\one_\La \in \{0, 1\}^\G$. The space $\Sub(\G)$ is compact and the group $\G$ acts continuously on $\Sub(\G)$ by conjugation:
$$
\G \times \Sub(\G)\to \Sub(\G), \qquad (g, \La)\mapsto g\La g^{-1}.
$$

An \emph{invariant random subgroup} (IRS) on $\G$ is a $\G$-invariant regular probability measure on $\Sub(\G)$.

Given an action $\G\act X$, we denote by $\Stab^0: X \to \Sub(\G)$ the \emph{open stabilizer map} $x\mapsto \G_x^0$. Let $\Xoo$ be the set of points in which $\Stab^0$ is continuous.

\begin{prop}\label{prop:Haus-germ}
Let $X$ be a compact $\G$-space. Then
\begin{enumerate}
\item[(i)]
 $\Xoo=\left(\bigcup_{g\in\G}\partial(\interior X^g)\right)^\mathsf{c}$ ;
\item[(ii)] If $\G$ is countable, then $\Xoo$ is a dense $G_\delta$ subset of $X$;
\item[(iii)] $\{x\in X: \G_x=\G_x^0\}\subset \Xoo$;
\item[(iv)] $\Xoo=X$ iff $\interior X^g$ is closed for all $g\in \G$.
\end{enumerate}
\end{prop}
\begin{proof}
(i) If $x\in \partial(\interior X^g)$ for some $g\in\G$, then $x\notin\interior X^g$ but there is a net $(x_i)\subset\interior X^g$ such that $x_i\to x$. Thus, $g\in\G_{x_i}^0$ for every $i$, but $g\notin\G_x^0$. This shows $\Stab^0$ is not continuous at $x$.

Conversely, if $\Stab^0$ is not continuous at a point $x\in X$, then there is $g\in \G$ such that the map $p_g\colon X\to\{0,1\}$ given by $p_g(z)=1\iff g\in\G_z^0$, for $z\in X$, is not continuous at $x$. Clearly, in this case, $p_g(x)=0$, i.e. $g\notin\G_x^0$. Furthermore, there is a net $y_i\in X$ such that $y_i\to x$, $g\in \G_{y_i}^0$ . Equivalently, $y_i\in \interior X^g$ for all $i$ and $x\notin \interior X^g$. Hence, $x\in \partial(\interior X^g)$.

(ii) follows immediately from (i) and the fact that the boundary of any open set has empty interior.

(iii) Let $x\notin  \Xoo$. By (i), there is $g\in\G$ such that $x\in\partial(\interior X^g)$. In particular, $x\in X^g\setminus\interior X^g$ and so $\G_x^0\subsetneq\G_x$.

(iv) Notice that, given $g\in\G$, $\interior X^g=(\Stab^0)^{-1}(\{\La\in\Sub(\G):g\in\G\})$. Hence, if $\Xoo=X$, then $\interior X^g$ is closed for all $g\in\G$.

Conversely, if $\interior X^g$ is closed for all $g\in\G$, then $\bigcup_{g\in\G}\partial(\interior X^g)=\emptyset$, and it follows from (i) that $\Xoo=X$. 
\end{proof}

The conditions in Proposition \ref{prop:Haus-germ}.(iv) are also equivalent to that the groupoid of germs of the action $\G\act X$ is Hausdorff (see section~\ref{section:groupoid} for definitions). Thus, following the terminology of \cite[Definition 2.9]{LeM}, we say the action $\G\act X$ \emph{has Hausdorff germs} if it satisfies the conditions in Proposition \ref{prop:Haus-germ}.(iv).

\begin{remark} 
Let $X$ be a compact $\G$-space.
\begin{enumerate}
\item[(i)] Let $\Stab\colon X\to \Sub(\G)$ be the map $x\mapsto\G_x$. In \cite{LeM}, Le Boudec and Matte Bon denoted the set of continuity points of this map by $X_0$. Furthermore, $X_0=\{x\in X:\G_x=\G_x^0\}$ (see e.g. \cite[Lemma 2.2]{LeM}). 
\item[(ii)] The set $\Xoo$ also received the attention of Nekrashevych in \cite{Nek}. In his terminology, the points in the complement of $\Xoo$ are \emph{purely non-Hausdorff singularities}.
\end{enumerate}
\end{remark}

\subsection{Group actions on $C^*$-algebras}
A unital $C^{*}$-algebra $A$ is called a \emph{$\G$-$C^{*}$-algebra} if there is an action $\G \act A$ of $\G$ on $A$ by $*$-automorphisms. 
We call a linear map $\phi\colon A\to B$ between $\G$-$C^{*}$-algebras a \emph{$\G$-map} if it is unital completely positive (ucp), and $\G$-equivariant, that is
\[
\phi(ga) = g\phi(a), \quad \forall\, g\in \G, ~ a\in A . 
\]

A $\G$-$C^*$-algebra $A$ is said to be \emph{$\G$-injective} if, given a completely isometric $\G$-map $\psi\colon B\to C$ and a $\G$-map $\varphi\colon B\to A$, there is a $\G$-map $\rho\colon C\to A$ such that $\rho\circ\psi=\varphi$.

A state $\tau$ on a $C^*$-algebra $A$ is a \emph{trace} if $\tau(ab) = \tau(ba)$ for all $a, b \in A$. 

A linear map $\psi\colon A\to B$ between $C^*$-algebras is said to be \emph{faithful} if, given $a\in A$, $\psi(a^*a)=0$ implies $a=0$.

For a probability measure $\nu$ on a compact $\G$-space $X$ we denote by $\cP_{\nu}$ its corresponding \emph{Poisson map}, i.e., the $\G$-map from $C(X)$ to $\ell^\infty(\G)$ defined by
\begin{equation*}
\cP_\nu(f)(g) =  \int g^{-1}fd\nu, \quad \forall g\in \G, ~ f\in C(X) . 
\end{equation*}
If $\nu$ is $\La$-invariant for a subgroup $\La$ of $\G$, then $\cP_{\nu}$ is mapped into $\ell^\infty(\G/\La)$.

If $\nu=\delta_x$ is a point measure for some $x\in X$, we denote the Poisson map simply by $\cP_x$.

\subsection{Unitary representations}
We denote the class of unitary representations of $\G$ by $\Rep(\G)$.
For $\pi\in \Rep(\G)$ we denote by $C^{*}_{\pi}(\G) := \overline{\operatorname{span}\{\pi(g) : g\in\G\}}^{\|\cdot\|}\subset B (\cH_\pi)$ the $C^*$-algebra generated by $\pi(\G)$ in $B (\cH_\pi)$, where $\cH_\pi$ is the Hilbert space of the representation $\pi$. The group $\G$ acts on $B (\cH_\pi)$ by inner automorphisms $g\cdot a := \pi(g)a\pi(g^{-1})$, $g\in \G$, $a\in B(\cH_\pi)$. 

An important class of unitary representations is the (left) quasi-regular representations $\lambda_{\G/\Lambda} \colon  \G \to \cU(\ell^2(\G/\Lambda))$ defined by
\[
(\lambda_{\G/\Lambda}(g)\xi)(h\Lambda) = \xi(g^{-1}h\Lambda)~~~~~~ \left(h\in\G, ~\xi\in \ell^2(\G/\Lambda)\right),\] 
where $\Lambda \leq \G$ is a subgroup. In the case of the trivial subgroup $\Lambda=\{e\}$, the $C^*$-algebra $C^{*}_{\lambda_\G}(\G)$ is called the \emph{reduced $C^*$-algebra} of $\G$. And for the choice of $\Lambda= \G$ the corresponding quasi-regular representation is the trivial representation of $\G$, which we denote by $1_\G$.

Given a compact $\G$-space $X$ and a Borel $\sigma$-finite quasi-invariant measure $\nu$ on $X$, the \emph{Koopman representation} $\k_\nu$ of $\G$ on $L^2(X, \nu)$ is defined by $\k_\nu(g)\xi(x) = \frac{dg\nu}{d\nu}(x)^{\frac12}\xi(g^{-1}x)$ for $g\in \G$ and $\xi\in L^2(X, \nu)$. Recall that $\nu$ is said to be quasi-invariant if $g\nu$ is in the same measure-class as $\nu$ for every $g\in\G$.

Let $\pi$ and $\sigma$ be two unitary representations of $\G$. We say $\pi$ is weakly contained in $\sigma$, written $\pi\prec\sigma$, if the map $\sigma(g)\mapsto\pi(g)$ extends to a $*$-homomorphism $C^*_\sigma(\G)\to C^*_\pi(\G)$, which then is obviously surjective. 

The Dirac function $\delta_e$ on $\G$ extends to a continuous trace on $C^{*}_{\lambda_\G}(\G)$, and consequently on $C^*_\pi(\G)$ for any $\pi\in\Rep(\G)$ that weakly contains $\la_\G$; we will refer to this trace as the \emph{canonical trace}.
The canonical trace on $C^{*}_{\lambda_\G}(\G)$ is faithful. In particular, the canonical trace is defined on $C^*_\pi(\G)$ iff $\pi$ weakly contains $\la_\G$.

\subsection{Covariant representations}
Given a compact $\G$-space $X$, a \emph{nondegenerate covariant representation} of $(\G,X)$ is a pair $(\pi,\rho)$, where $\pi\in\Rep(\G)$ and $\rho\colon C(X)\to B(\cH_\pi)$ is a unital $\G$-equivariant $*$-homomorphism. In this case, we let $C^*_{\pi\times\rho}(\G,X):=\overline{\mathrm{span}}\{\rho(f)\pi(g):f\in C(X),g\in\G\}$. 

If $(\pi_1,\rho_1)$ and $(\pi_2,\rho_2)$ are nondegenerate covariant representations of $(\G,X)$, we say that $(\pi_1,\rho_1)$ is \emph{weakly contained} in $(\pi_2,\rho_2)$, and write $(\pi_1,\rho_1)\prec(\pi_2,\rho_2)$ if there is a $*$-homomorphism $C^*_{\pi_2\times\rho_2}(\G,X)\to C^*_{\pi_1\times\rho_1}(\G,X)$ which maps $\rho_2(f)\pi_2(g)$ to $\rho_1(f)\pi_1(g)$, for every $g\in\G$ and $f\in C(X)$.
\subsection{Amenability}
Recall that a discrete group $\G$ is amenable if $\ell^\infty(\G)$ admits a (left) translation invariant state. 
Amenability can also be characterized in terms of weak containment. The group $\G$ is amenable iff $1_\G\prec\lambda_\G$, iff $\pi\prec\lambda_\G$ for every unitary representation $\pi$ of $\G$. A subgroup $\Lambda\leq \G$ is amenable iff $\lambda_{\G/\Lambda}\prec\lambda_\G$.

Let $\Lambda\leq \G$ be a subgroup. We say $\Lambda$ is \emph{co-amenable} in $\G$ if $1_\G \prec \lambda_{\G/\Lambda}$, which is equivalent to the existence of a $\G$-invariant state $\ell^\infty(\G/\Lambda)\to\bC$.   

\begin{proposition}\label{prop:co-am}
Given groups $\La_1\leq\La_2\leq\G$, we have that $\La_1$ is co-amenable in $\La_2$ iff $\la_{\G/\La_2}\prec\la_{\G/\La_1}$. 
\end{proposition}
\begin{proof}

If $\La_1$ is co-amenable in $\La_2$, then $$\la_{\G/\La_2} = \Ind_{\La_2}^\G(1_{\La_2})\prec\Ind_{\La_2}^\G(\la_{\La_2/\La_1})= \la_{\G/\La_1}.$$

Conversely, if $\la_{\G/\La_2}\prec\la_{\G/\La_1}$, then there is a $\G$-map from $\ell^\infty(\G/\La_1)$ to $\ell^\infty(\G/\La_2)$, and therefore there is a $\La_2$-invariant state on $\ell^\infty(\G/\La_1)$. Composing this state with a $\La_2$-map from $\ell^\infty(\La_2/\La_1)$ to $\ell^\infty(\G/\La_1)$ yields a $\La_2$-invariant state on $\ell^\infty(\La_2/\La_1)$. Hence, $\La_1$ is co-amenable in $\La_2$.
\end{proof}

\subsection{Boundary actions}
\label{SS-BoudaryActions}
An action of a group $\G$ on a compact space $X$ is said to be \emph{strongly proximal} if, for every probability $\nu\in\pr(X)$ on $X$, the weak* closure of the orbit $\G\nu$ contains some point measure $\delta_x$, $x\in X$.
The action $\G\act X$ is called a \emph{boundary action} (or $X$ is a \emph{$\G$-boundary}) if it is both minimal and strongly proximal. 

The action $\G\act X$ is called an \emph{extreme boundary action} if $|X|>2$ and for any closed set $C\subsetneq X$ and any open set $\emptyset\neq U\subset X$, there is $g\in \G$ such that $g(C)\subset U$. By \cite[Theorem 2.3]{Gla}, any extreme boundary action is a boundary action.

By \cite{Furst73}, every group $\G$ admits a universal boundary $\partial_F \G$, which we call the \emph{Furstenberg boundary}: $\partial_F \G$ is a $\G$-boundary and every $\G$-boundary is a continuous $\G$-equivariant image of $\partial_F \G$. In fact, for every $\G$-boundary $X$ there is a unique continuous $\G$-equivariant map from $\partial_F \G$ onto $X$, which we will denote by $\b_{\hspace{-0.1em} X}\colon \partial_F \G\to X$. Moreover, any continuous $\G$-equivariant map from $\partial_F \G$ into $\pr(X)$ is mapped onto the set of Dirac measures, hence coming from $\b_{\hspace{-0.1em} X}$ (\cite[Proposition 4.2]{Furst73}). This is equivalent to saying that for any $\G$-boundary $X$ there is a unique $\G$-map from $C(X)$ to $C(\partial_F \G)$, which is the embedding coming from $\b_{\hspace{-0.1em} X}$. We will use this fact frequently in this paper, and we will also refer to it by saying that $C(\partial_F\G)$ is \emph{rigid} (this notion of ridigity is related to, but different from the notion of $\G$-rigid extension introduced in \cite{Ham85}). Furthermore, $C(\partial_F\G)$ is $\G$-injective (\cite[Theorem 3.11]{KalKen}).

\begin{defn}\label{def:bnd-map}
Let $A$ be a $\G$-$C^*$-algebra. A \emph{boundary map} on $A$ is a $\G$-map $\psi \colon A\to C(\partial_F\G)$.
\end{defn}

We recall that $C^*$-simple groups have been characterized in \cite{KalKen} in terms of their actions on boundaries 
as follows (see also \cite{BKKO}).
\begin{thm}[\cite{KalKen}] \label{Theo-KK}
Let $\G$ be a discrete group. The following properties are equivalent:
\begin{itemize}
\item[(i)] $\G$ is $C^*$-simple;
\item [(ii)] the action $\G\act \partial_F\G$ is free;
\item [(iii)] there exists a topologically free boundary action $\G\act X$.
\end{itemize}
\end{thm}


Generally, every group $\G$ has a largest amenable normal subgroup, called the \emph{amenable radical} $\Rad(\G)$. It was proved by Furman in \cite{Furm} that $\Rad(\G)$ coincides with the kernel of the action of $\G$ on its Furstenberg boundary $\partial_F \G$.

We also recall the following description of traces on $C^*_{\la_\G}(\G)$ from \cite{BKKO}.
\begin{thm}[\cite{BKKO}] \label{BKKO-tr}
Every trace on $C^*_{\la_\G}(\G)$ is supported on the amenable radical $\Rad(\G)$ of $\G$.
\end{thm}

This yields a complete characterization of groups with the unique trace property. 
\begin{thm} [\cite{BKKO}]\label{BKKO-unqtr}
Let $\G$ be a discrete group. The following properties are equivalent:
\begin{itemize}
\item[(i)] the canonical trace is the unique trace on $C^*_{\la_\G}(\G)$;
\item [(ii)] the action $\G\act \partial_F \G$ is faithful;
\item [(iii)] there exists a faithful boundary action $\G\act X$.
\end{itemize}
\end{thm}
A long well-known fact (see \cite[Lemma p.289]{BCH}) is that a group with a non-trivial amenable radical is not $C^*$-simple; examples of non $C^*$-simple groups with a trivial amenable radical have been given in \cite{LeBou}.

\medskip

We will need the following result from \cite{BKKO}. As observed in \cite[Remark 3.4]{Urs}, there is a minor inaccuracy in the proof of this result in \cite{BKKO}. For the convenience of the reader, we give the complete argument here.
\begin{lemma}[{\cite[Lemma 3.2]{BKKO}}]\label{lem:minint}
A continuous equivariant map $\pi\colon Y\to X$ between minimal compact $\G$-spaces sends sets of non-empty interior to sets of non-empty interior.
\end{lemma}

\begin{proof}
Let $U\subset Y$ be a set of non-empty interior and take $V\subset U$ a non-empty open subset such that $\overline{V}\subset U$. By minimality and compactness, we have that $Y=\bigcup_{g\in F}gV$ for some finite subset $F\subset \G$.  Hence $X=\bigcup_{g\in F}g\pi(\overline{V})$.  By Baire's theorem, we conclude that $\pi(\overline{V})\subset\pi(U)$ has non-empty interior.
\end{proof}


\section{Boundary actions and unitary representations}\label{sec:gen}

We now begin our study of boundary maps. In particular, given a $\G$-boundary $X$ and $\pi\in\Rep(\G)$, we investigate $\G$-maps from $C(X)$ to $B(\cH_\pi)$ and from $C^*_\pi(\G)$ to $C(\partial_F\G)$, and the connections to the subgroup structures of $\G$. 

The idea of utilizing the dynamics of boundary actions in describing noncommutative boundary maps originated in \cite{BKKO}, \cite{Haag15} for the special case of traces, and in \cite{Ken15} for general boundary maps on $C^*_{\la_\G}(G)$. In this paper we are interested in more general class of $C^*$-algebras generated by $\G$, and will extend the techniques in order to cover these cases. 

But we begin in this section with gathering some of the key techniques and observations, which are known to experts and have been used in different forms in the above mentioned works as well as others that followed them. Indeed, the main purpose of the section is to formulate an abstract framework in which these techniques can be used in order to provide more clarity on the existing ideas which we improve upon in the proofs of our main results in later sections.

\medskip

A key observation regarding the connection of boundary maps to $C^*$-simplicity is the following.
\begin{proposition}\label{prop:ideal}
Given $\pi\in\Rep(\G)$ and a boundary map $\psi$ on $C^*_\pi(\G)$, let $I_\psi:=\{a\in C^*_\pi(\G):\psi(a^*a)=0\}$. Then $I_\psi$ is an ideal of $C^*_\pi(\G)$ and for every proper ideal $J\unlhd C^*_\pi(\G)$, there is a boundary map $\psi$ on $C^*_\pi(\G)$ such that $J\subset I_\psi$.

In particular, if $ C^*_\pi(\G)$ admits a unique boundary map $\psi$, then $I_\psi$ contains all proper ideals of $C^*_\pi(\G)$, and $C^*_\pi(\G)/I_\psi$ is the unique non-zero simple quotient of $C^*_\pi(\G)$.
\end{proposition}

\begin{proof}
That $I_\psi$ is a left ideal of $C^*_\pi(\G)$ follows from the Schwarz inequality for ucp maps. We also see from $\G$-equivariance that for $a\in I_\psi$ and $g\in\G$, $$\psi(\pi(g^{-1})a^*a\pi(g)) = g^{-1}\cdot\psi(a^*a) = 0,$$ which shows $I_\psi$ is also a right ideal. Now given a proper ideal $J\unlhd C^*_\pi(\G)$, the $\G$-action naturally descends to the quotient $C^*_\pi(\G)/J$. By $\G$-injectivity, there is a $\G$-map $C^*_\pi(\G)/J \to C(\partial_F\G)$. Composing the latter map with the canonical quotient map $C^*_\pi(\G)\to C^*_\pi(\G)/J$, we get a $\G$-map $\psi\colon C^*_\pi(\G)\to C(\partial_F\G)$ such that $J\subset I_\psi$.

The remaining assertions are immediate from the above.
\end{proof}

The following is one of the key properties which we will exploit in this paper. In fact, this result gives one half of the description of boundary maps we prove in Theorem~\ref{thm:uniquemap-koop}.
The result is essentially known, has been used before in other works, and in the case of covariant representations follows from techniques of \cite{BKKO}. 

We recall that for a $\G$-boundary $X$ and $g\in\G$ we denote by $\Delta_g$ the set $\overline{\b_{\hspace{-0.1em} X}^{-1}(\interior X^g)}$.

\begin{prop}\label{prop:supp}
Let $X$ be a $\G$-boundary and $\pi\in\Rep(\G)$ such that there exists a $\G$-map $\rho\colon C(X)\to B(H_\pi)$. Given a $\G$-map $\psi\colon C^*_\pi(\G)\to C(\partial_F\G)$, we have that $$\supp\psi(\pi(g))\subset\Delta_g$$ for every $g\in\G$.
\end{prop}

\begin{proof}
Fix $g\in \G$ and suppose that there is $y\notin \Delta_g$ such that $\psi(\pi(g))(y)\neq 0$. Then there is a clopen neighborhood $U$ of $y$ such that $U\cap \b_{\hspace{-0.1em} X}^{-1}(\interior X^g)=\emptyset$ and 
\begin{equation}\label{eq:piz}
\psi(\pi(g))(z)\neq 0
\end{equation}
 for any $z\in U$.

By Lemma \ref{lem:minint}, $\b_{\hspace{-0.1em} X}(U)$ has non-empty interior. Since $$\b_{\hspace{-0.1em} X}(U)\cap\interior X^g=\emptyset,$$ there is $u\in U$ such that $\b_{\hspace{-0.1em} X}(u)\notin X^g$. Let $f\in C(X)$ such that $0\leq f\leq 1$, $f(\b_{\hspace{-0.1em} X}(u))=1$ and $gf(\b_{\hspace{-0.1em} X}(u))=0$.

Let $\tilde{\psi}\colon B
(H_\pi)\to C(\partial_F\G)$ be a $\G$-map extending $\psi$. Notice that $\tilde{\psi}(\rho(f))(u)=f(\b_{\hspace{-0.1em} X}(u))=1$ and $$\tilde{\psi}(\pi(g)\rho(f)\pi(g)^*)(u)=\tilde{\psi}(\rho(gf))(u)=gf(\b_{\hspace{-0.1em} X}(u))=0.$$

Applying \cite[Lemma 2.2]{HartKal} to the state $\delta_u\circ\tilde{\psi}$ on $B(H_\pi)$, we conclude that $\psi(\pi(g))(u)=0$, which contradicts \eqref{eq:piz}.
\end{proof}

The following consequence is a folklore among the experts, we record the statement for future reference.

\begin{corollary}\label{cor:trace-supp-N}
Let $X$ be a $\G$-boundary and $\pi\in\Rep(\G)$ such that there exists a $\G$-map $\rho\colon C(X)\to B(H_\pi)$. Then every trace on $C^*_{\pi}(\G)$ is supported on $\ker(\G\act X)$.
\end{corollary}

\begin{proof}
Let $\tau$ be a trace on $C^*_\pi(\G)$. We consider $\tau$ as a $\G$-map from $C^*_{\pi}(\G)$ to $C(\partial_F\G)$ whose image consists of constant functions on $C(\partial_F\G)$. 

Let $g\notin \ker(\G\act X)$. In particular, $\Delta_g \neq \partial_F\G$, which implies by Proposition~\ref{prop:supp} that $\supp \tau(g) \neq \partial_F\G$, hence $\tau(\pi(g))=0$.
\end{proof}

\begin{remark}
Given $\pi\in\Rep(\G)$, let $\cB^\pi$ be the collection of $\G$-boundaries $X$ for which there exists a $\G$-map $\rho\colon C(X)\to B(H_\pi)$. In view of Theorem~\ref{BKKO-tr} and Corollary \ref{cor:trace-supp-N}, the normal subgroup
\[
N_\pi:= \bigcap_{X\in\cB^\pi}\ker(\G\act X)  \le \G
\]
should be considered as the ``$\pi$-amenable radical'' of $\G$. In fact, we have $N_{\la_\G} =\Rad(\G)$. 
Another notion of $\pi$-amenable radical was considered in \cite[Definition 4.6]{BearKal} as the kernel ${\Rad}_\pi(\G)$ of the Furstenberg-Hamana boundary $\mathcal{B}_\pi$ of $\pi$ (\cite[Definition 3.6]{BearKal}). For any $\pi\in\Rep(\G)$, we have ${\Rad}_\pi(\G)\subset N_\pi$. In fact, there is a ${\Rad}_\pi(\G)$-invariant state on $B(H_\pi)$, hence on $C(X)$ for any $\G$-boundary $X$ with a $\G$-map $C(X)\to B(H_\pi)$. Since ${\Rad}_\pi(\G)$ is normal in $\G$, by strong proximality it follows ${\Rad}_\pi(\G)$ acts trivially on $X$.
\end{remark}

The following observations are some simple illustrations of the relevance of these results in structural properties of $C^*$-algebras.

\begin{proposition}\label{prop:nuclear}
Let $X$ be a compact $\G$-space and $\pi\in\Rep(\G)$ such that there exists a $\G$-map $\rho\colon C(X)\to B(H_\pi)$. If $C^*_\pi(\G)$ is nuclear and admits trace, then $X$ admits a $\G$-invariant probability.
\end{proposition}

\begin{proof}
This follows from the well-known fact that any trace on a nuclear $C^*$-algebra is amenable. 
\end{proof}

\begin{corollary}
Let $\G$ be a group and $\La\in\Sub(\G)$ such that $C^*_{\la_{\sfrac{\G}{\La}}}\!(\G)$ is nuclear. Then $C^*_{\la_{\sfrac{\G}{\La}}}\!(\G)$ admits a trace iff $1_\G\prec\la_{\G/\La}$.
\end{corollary}

\begin{proof}
The assertion follows immediately by applying Proposition \ref{prop:nuclear} to the Stone-\v{C}ech compactification $X=\beta(\G/\La)$.
\end{proof}

In the remaining of the section, we propose an abstract language to provide a general framework where the above techniques can be applied. We also take into consideration the connection to the subgroup structure of $\G$ following the ideas explored in \cite{Ken15} and \cite{BKKO}.

\begin{defn}
By an \emph{action class of $\G$} we mean a collection $\cC$ of compact $\G$-spaces (e.g. all faithful or topologically free $\G$-boundaries). 

A unitary representation $\pi\in \Rep(\G)$ is said to be a $\cC$-representation, where $\cC$ is an action class of $\G$, if there is a $\G$-map from $C(X)$ to $B(\cH_\pi)$ for some $X\in\cC$. We denote by $\Rep_\cC(\G)$ the collection of all $\cC$-representations of $\G$. 
\end{defn}

\begin{example}
Let $X\in \cC$, and let $\nu$ be a quasi-invariant $\sigma$-finite measure on $X$. Then the Koopman representation $\k_\nu$ of $\G$ on $L^2(X, \nu)$ is a $\cC$-representation.
\end{example}

\begin{prop}\label{prop:C-Reps}
If $\pi\in \Rep_\cC(\G)$ then $\sigma\in \Rep_\cC(\G)$ for every $\sigma\prec\pi$.
\end{prop}

\begin{proof}
Let $\pi\in \Rep_\cC(\G)$. So there is a $\G$-map $\psi\colon  C(X)\to B(\cH_\pi)$ for some $X\in\cC$. Suppose $\sigma\prec\pi$. The $*$-homomorphism $C^*_\pi(\G)\to C^*_\sigma(\G)$ extends to a $\G$-map $\varphi\colon  B(\cH_\pi)\to B(\cH_\sigma)$. The composition $\psi\circ \varphi \colon  C(X)\to B(\cH_\sigma)$ is a $\G$-map. Thus, $\sigma\in \Rep_\cC(\G)$.
\end{proof}

Denote by $\cC_{\scriptscriptstyle\tiny{\rm bnd}}^{\scriptscriptstyle\tiny{\rm fth}}$ and $\cC_{\scriptscriptstyle\tiny{\rm bnd}}^{\scriptscriptstyle\tiny{\rm tfr}}$ the action class of all faithful, and all topologically free $\G$-boundaries, respectively.

\begin{corollary}\label{cor:C-Reps}
\begin{enumerate}
\item[(i)]
Let $\pi\in \Rep_{\cC_{\scriptscriptstyle\tiny{\rm bnd}}^{\scriptscriptstyle\tiny{\rm fth}}}(\G)$. Then either $C^*_{\pi}(\G)$ admits no trace, or otherwise $\pi$ weakly contains $\la_\G$ and the canonical trace is the unique trace on $C^*_{\pi}(\G)$.
\item[(ii)]
Let $\pi\in \Rep_{\cC_{\scriptscriptstyle\tiny{\rm bnd}}^{\scriptscriptstyle\tiny{\rm tfr}}}(\G)$. Then the canonical trace is the unique boundary map on $C^*_\pi(\G)$. In particular, $\pi$ weakly contains $\la_\G$, and $C^*_{\pi}(\G)$ has a proper ideal containing all proper ideals of $C^*_{\pi}(\G)$. 
\end{enumerate}
\end{corollary}

\begin{proof}
(i) If $\pi\in \Rep_{\cC_{\scriptscriptstyle\tiny{\rm bnd}}^{\scriptscriptstyle\tiny{\rm fth}}}(\G)$, then $N_\pi$ is trivial, hence if $C^*_{\pi}(\G)$ admits a trace, it must coincide with the canonical trace by Corollary~\ref{cor:trace-supp-N}, and in particular $\la_\G\prec\pi$.

(ii) The first assertion is an immediate consequence of Proposition~\ref{prop:supp} and the definition of topologically freeness. That $C^*_\pi(\G)$ has a proper ideal containing all proper ideals follows directly from Proposition \ref{prop:ideal}. 
\end{proof}

\begin{defn}
Let $\cC$ be an action class of $\G$. We say $\Lambda\in \Sub(\G)$ is a $\cC$-subgroup if $\lambda_{\G/\Lambda} \in \Rep_\cC(\G)$. We denote by $\Sub_\cC(\G)$ the set of all $\cC$-subgroups of $\G$.
\end{defn}

\begin{prop}\label{prop:C-subgrps}
Let $\cC$ be an action class of $\G$ and $\Lambda\in \Sub(\G)$. Then $\Lambda\in \Sub_\cC(\G)$ if and only if $\Lambda$ fixes a probability on some $X\in \cC$.
\end{prop}

\begin{proof}
$(\Rightarrow)\colon  $ let $\psi\colon  C(X)\to B(\ell^2(\G/\Lambda))$ be a $\G$-map where $X\in\cC$. Then the restriction of the $\Lambda$-invariant state on $B(\ell^2(\G/\Lambda))$ to $\psi(C(X))$ yields a $\Lambda$-invariant probability on $X$.\\
$(\Leftarrow)\colon  $ let $X\in \cC$ be such that there is a $\Lambda$-invariant $\nu\in\pr(X)$. The Poisson map $\cP_\nu\colon  C(X) \to \ell^\infty(\G/\La)\subset B(\ell^2(\G/\Lambda))$ is the desired $\G$-map.
\end{proof}

\begin{remark}
\begin{enumerate}
\item[(i)]
$\Sub_{\cC_{\scriptscriptstyle\tiny{\rm bnd}}^{\scriptscriptstyle\tiny{\rm tfr}}}(\G)$ coincides with the class $\Sub_{wp}(\G)$ of weakly parabolic subgroups in the sense of \cite[Definition 6.1]{BekKal}.

\item[(ii)]
Since $\lambda_{\G/\Lambda} \sim_u \lambda_{\G/g\Lambda g^{-1}}$ for every $g\in\G$, by Proposition~\ref{prop:C-Reps}, the set $\Sub_\cC(\G)$ is $\G$-invariant for any $\cC$ (cf. \cite[Remark 6.2]{BekKal}).

\item[(iii)]
By results of \cite{KalKen} and \cite{BKKO}, $\Sub_{\cC_{\scriptscriptstyle\tiny{\rm bnd}}^{\scriptscriptstyle\tiny{\rm fth}}}(\G) \neq\emptyset$ iff $\G$ has trivial amenable radical, and $\Sub_{\cC_{\scriptscriptstyle\tiny{\rm bnd}}^{\scriptscriptstyle\tiny{\rm tfr}}}(\G) \neq\emptyset$ iff $\G$ is $C^*$-simple. Once non-empty, these sets contain $\Sub_{am}(\G)$, the space of all amenable subgroups of $\G$.

\item[(iv)]
It follows from Proposition \ref{prop:C-subgrps} that for $\Lambda_1\leq \Lambda_2\leq \G$, with $\Lambda_1$ co-amenable in $\Lambda_2$, if $\Lambda_1\in \Sub_{\cC_{\scriptscriptstyle\tiny{\rm bnd}}^{\scriptscriptstyle\tiny{\rm *}}}(\G)$ then $\Lambda_2\in \Sub_{\cC_{\scriptscriptstyle\tiny{\rm bnd}}^{\scriptscriptstyle\tiny{\rm *}}}(\G)$ for ${\rm * =fth, tfr}$.

For the same reason, given $\Lambda_1,\La_2 \leq \G$ with $\La_1$ co-amenable to $\La_2$ relative to $\G$ in the sense of \cite[7.C]{CapraceMonod}), we have that if $\Lambda_1\in \Sub_{\cC_{\scriptscriptstyle\tiny{\rm bnd}}^{\scriptscriptstyle\tiny{\rm *}}}(\G)$ then $\Lambda_2\in \Sub_{\cC_{\scriptscriptstyle\tiny{\rm bnd}}^{\scriptscriptstyle\tiny{\rm *}}}(\G)$ for ${\rm * =fth, tfr}$.

\item[(v)]
Similar arguments as in the proof of \cite[Theorem 4.3]{Haag15} imply that if $\La\in\Sub_{\cC_{\scriptscriptstyle\tiny{\rm bnd}}^{\scriptscriptstyle\tiny{\rm fth}}}(\G)$, then $0\in\overline{\mathrm{conv}}\{\la_{\G/\La}(sgs^{-1}):s\in\G\}$ for all $g\in\G\setminus\{e\}$.
\end{enumerate}

\end{remark}

Since $\Sub_{\cC_{\scriptscriptstyle\tiny{\rm bnd}}^{\scriptscriptstyle\tiny{\rm tfr}}}(\G) \subset \Sub_{\cC_{\scriptscriptstyle\tiny{\rm bnd}}^{\scriptscriptstyle\tiny{\rm fth}}}(\G)$, the following generalizes \cite[Proposition 6.5]{BekKal}, and the proof is exactly the same.

\begin{prop}\label{Prop-Sub_g}
$\Sub_{\cC_{\scriptscriptstyle\tiny{\rm bnd}}^{\scriptscriptstyle\tiny{\rm fth}}}(\G)$ contains no non-trivial normal subgroup of $\G$.
\end{prop}

In \cite[Corollary 6.8]{BekKal} it was proved that $\Sub_{\cC_{\scriptscriptstyle\tiny{\rm bnd}}^{\scriptscriptstyle\tiny{\rm tfr}}}(\G)$ contains no recurrent subgroups, generalizing Kennedy's characterization of $C^*$-simplicity \cite[Theorem 1.1]{Ken15}.

The similar characterization for unique trace property is the following: $\G$ has unique trace property iff $\G$ has no non-trivial amenable IRS (\cite[Corollary 1.5]{BDL} and \cite[Theorem 1.3]{BKKO}).
Thus, the following question is natural:

\begin{question}
Is every IRS with support contained in $\Sub_{\cC_{\scriptscriptstyle\tiny{\rm bnd}}^{\scriptscriptstyle\tiny{\rm fth}}}(\G)$ trivial?
\end{question}

Given $\pi\in \Rep(\G)$, let $\cW(\pi):=\{H\leq\G:\la_{\G/H}\prec\pi\}$. Notice that $\cW(\pi)$ is a closed $\G$-invariant subset of $\Sub(\G)$. For example, $\cW(\{\lambda_\G\})=\Sub_{\rm am}(\G)$, the set of all amenable subgroups of $\G$ and, in general, $\cW(\la_{\G/\La})$ contains all subgroups $L\leq\G$ such that $\La\leq L$ and $\La$ is co-amenable in $L$ (Proposition \ref{prop:co-am}).

\begin{theorem}\label{thm:irs}
Let $\pi\in \Rep(\G)$. Then any IRS supported on $\cW(\pi)$ is supported on $N_\pi$.
\end{theorem}

\begin{proof}
Let $\eta$ be an IRS supported on $\cW(\pi)$. Let $\varphi_\eta$ be the positive definite function on $\G$ given by $\varphi_\eta(g)=\eta(\{L:g\in L\})$ for $g\in\G$, and $\la_\eta$ be the GNS representation associated to $\varphi_\eta$ (see e.g. \cite[Lemma 2.3]{HartKal} for a proof that $\varphi_\eta$ is positive definite). 

Since $\lambda_{\G/H}\prec\pi$ for all $H\in  \cW(\pi)$, and $\eta\in\overline{\mathrm{conv}}^{w^*}\{\delta_H:H\in\cW(\pi)\}$, we conclude that $\la_\eta\prec\pi$. Since $\eta$ is an IRS, the state on $C^*_{\la_\eta}(\G)$ associated to $\varphi_\eta$ is a trace. Finally, Corollary \ref{cor:trace-supp-N} implies that $\eta$ is supported on $N_\pi$.
\end{proof}

\begin{corollary}
Given $\La\in \Sub_{\cC_{\scriptscriptstyle\tiny{\rm bnd}}^{\scriptscriptstyle\tiny{\rm fth}}}(\G)$, any IRS supported on $\cW(\la_{\G/\La})$ is trivial.
\end{corollary}

\begin{remark}
We recall again that $\cW(\pi)$ is a compact $\G$-space, and we observed in the proof of Theorem \ref{thm:irs} that if $\cW(\pi)$ admits a $\G$-invariant probability, then $C^*_{\pi}(\G)$ admits a trace.
The converse to this is not true: take any unital $C^*$-algebra $A$ admitting a trace, and choose a group $\G$ of unitaries of $A$ that generates $A$ and contains -$1_A$. This gives a unitary representation of $\G$ which admits a trace, and such that none of the traces that it admits come from an IRS.
\end{remark}

We end the section with an observation which shows how the boundary techniques reviewed in this section can be used in the context of a different type of rigidity problem.

Denote by $\Rep_{{\rm II}_1}(\G)$ the collection of all representations $\pi$ of $\G$ such that $\pi(\G)''\subset B(H_\pi)$ is a ${\rm II}_1$-factor. Recall that $\G$ is said to be \emph{operator algebraic superrigid} if for any $\pi\in \Rep_{{\rm II}_1}(\G)$ the map $\pi(g)\mapsto \lambda_\G(g)$ extends to a von Neumann isomorphism $\pi(\G)''\cong {\rm L}\G$.
\begin{proposition}
A non-amenable group $\G$ is operator algebraic superrigid iff $$\Rep_{{\rm II}_1}(\G) \subset \Rep_{\cC_{\scriptscriptstyle\tiny{\rm bnd}}^{\scriptscriptstyle\tiny{\rm fth}}}(\G).$$ 
\end{proposition}

\begin{proof}
$(\Leftarrow) $ Let $\pi\in\Rep_{{\rm II}_1}(\G)$. By part (i) of Corollary \ref{cor:C-Reps}, we have that the trace on $\pi(\G)''$ is the canonical one. 
Clearly, this implies that $\pi(\G)''$ is canonically isomorphic to ${\rm L}\G$.\\
$(\Rightarrow)  $ Since every non-amenable group has a ${\rm II}_1$-factorial representation, it follows $\G$ is just-non-amenable, i.e. all non-trivial normal subgroups are co-amenable. This implies that the amenable radical of $\G$ is trivial, hence $\G\act \partial_F\G$ is faithful by \cite{Furm}. In particular, $\la_\G\in\Rep_{\cC_{\scriptscriptstyle\tiny{\rm bnd}}^{\scriptscriptstyle\tiny{\rm fth}}}(\G)$.
Let $\pi\in\Rep_{{\rm II}_1}(\G)$. Then $\lambda_g\mapsto \pi(g)$ extends to a $C^*$-isomorphism $C^*_{\lambda_\G}(\G)\to C^*_\pi(\G)$. It follows from Proposition \ref{prop:C-Reps} that $\pi\in\Rep_{\cC_{\scriptscriptstyle\tiny{\rm bnd}}^{\scriptscriptstyle\tiny{\rm fth}}}(\G)$.
\end{proof}

\section{Uniqueness of boundary maps}\label{section: qrb}

In this section, given a group $\G$, we study a class of representations $\pi\in\Rep(\G)$ for which $C^*_\pi(\G)$ admits a unique boundary map. We are particularly interested in the case of quasi-regular representations.

\begin{definition}\label{grprep}
Let $X$ be a compact $\G$-space. A \emph{germinal representation} of $(\G,X)$ is a nondegenerate covariant representation $(\pi,\rho)$ of $(\G,X)$ such that 
\begin{equation}\label{eq:cond}
\pi(g)\rho(f)=\rho(f),\quad \forall g\in\G, ~ f\in C(X) \text{ with } \supp f\subset  X^g .
\end{equation}
In this case, we also say that $\pi$ is a germinal representation of $\G$ (relative to $X$).
\end{definition}

\begin{remark}
Let $X$ be a compact $\G$-space. Given $g\in \G$, notice that the set $\{f\in C(X):\supp f\subset\interior X^g\}$ is dense in the set $\{f\in C(X):\supp f\subset X^g\}$. Therefore, a nondegenerate covariant representation $(\pi,\rho)$ of $(\G,X)$ is a germinal representation iff 
$$\pi(g)\rho(f)=\rho(f),\quad \forall g\in\G, ~ f\in C(X) \text{ with } \supp f\subset \interior X^g .$$ 
\end{remark}

\begin{proposition}\label{prop:grpoi-reps-exam}
Let $X$ be a compact $\G$-space. The following are germinal representations of $(\G,X)$:

\begin{enumerate}
\item[(i)] The pair $(\k_\nu,\rho)$, where $\nu$ is a $\sigma$-finite quasi-invariant measure on $X$, $\k_\nu$ is the Koopman representation of $\G$ on $L^2(X, \nu)$ and $\rho\colon C(X)\to B(L^2(X,\nu))$ is the representation by multiplication operators. 
\item[(ii)] The pair $(\la_{\G/H},\cP_x)$, where $x\in X$, $H\in\Sub(\G)$ is such that $\G_x^0\leq H\leq \G_x$, and $\cP_x\colon C(X)\to B(\ell^2(\G/H))$ is the Poisson map.
\end{enumerate}

\end{proposition}

\begin{proof}
(i) This follows from the fact that, given $g\in\G$, $\frac{dg\nu}{d\nu}\one_{X^g}=\one_{X^g}$.

(ii) Fix $g\in\G$ and $f\in C(X)$ with $\supp f\subset X^g$. Given $\delta_{kH}\in\ell^2(\G/H)$, we have
\begin{align*}
\cP_x(f)\delta_{kH}=
\begin{cases}
f(kx)\delta_{kH}, & \text{if $kx\in\interior X^g$}\\
0, & \text{otherwise}.\end{cases}
\end{align*}
On the other hand, given $k\in\G$ such that $kx\in\interior X^g$, we have that $k^{-1}gk\in\G_x^0$, hence $gkH=kH$. This concludes the proof that $\la_{\G/H}(g)\cP_x(f)=\cP_x(f)$.
\end{proof}

We are now ready to prove the main result of this section, the uniqueness of boundary maps on $C^*$-algebras of germinal representations of boundary actions.
\begin{theorem}\label{thm:uniquemap-koop}
Let $X$ be a $\G$-boundary and $(\pi,\rho)$ a germinal representation of $(\G,X)$. Then there is a unique boundary map $\psi$ on $C^*_{\pi\times\rho}(\G,X)$, and $\psi|_{C^*_\pi(\G)}$ is the unique boundary map on $C^*_\pi(\G)$.

Furthermore, $\psi({\pi}(g))=\one_{\Delta_g}$ and ${\psi}(\rho(f))=f\circ\b_{\hspace{-0.1em} X}$  for all $g\in \G$ and $f\in C(X)$.
\end{theorem}

\begin{proof}
By $\G$-injectivity of $C(\partial_F\G)$, there exists a boundary map ${\psi}\colon C^*_{\pi\times\rho}(\G,X)\to C(\partial_F\G)$. Since any boundary map on $C^*_\pi(\G)$ can be extended to a boundary map on $ C^*_{\pi\times\rho}(\G,X)$, uniqueness of $\psi$ would imply uniqueness of its restriction on $C^*_\pi(\G)$.

By rigidity of $C(\partial_F\G)$ in the sense of Section \ref{SS-BoudaryActions}, we have that ${\psi}(\rho(f))=f\circ\b_{\hspace{-0.1em} X}$, for $f\in C(X)$.
 In particular, $\rho(C(X))$ is contained in the multiplicative domain of ${\psi}$ and therefore, ${\psi}$ is uniquely determined by $\psi|_{C^*_\pi(\G)}$. 

Fix $g\in \G$, and we will show that $\psi({\pi}(g))=\one_{\Delta_g}$. By Proposition~\ref{prop:supp}, we only need to show that $\psi({\pi}(g))$ is constant $1$ on $\Delta_g$.

Given $y\in \b_{\hspace{-0.1em} X}^{-1}(\interior X^g)$, take $f\in C(X)$ such that $f(\b_{\hspace{-0.1em} X}(y))=1$, and $\mathrm{supp}f\subset X^g$. 

By \eqref{eq:cond}, we have that $\psi({\pi}(g)){\psi}(\rho(f))={\psi}(\rho(f))$. By applying both sides of this equation to $y$, we conclude that $\psi({\pi}(g))(y)=1$. By continuity, $\psi({\pi}(g))(z)=1$ for any $z\in \Delta_g$.
\end{proof}

We proceed with a list of corollaries of our above results in which we recover most of the main previously proven results on $C^*$-simplicity.

We begin with the original characterization of $C^*$-simplicity from \cite{KalKen}.

\begin{corollary}[{\cite[Theorem 6.2]{KalKen}}]\label{classical}
A group $\G$ is $C^*$-simple iff it has a topologically free boundary.
\end{corollary}

\begin{proof}
Suppose $\G$ is $C^*$-simple. Let $x\in \partial_F \G$. By $\G$-injectivity of $C(\partial_F \G)$, $\G_x$ is amenable, hence $\csx = C^*_{\la_{\G}}(\G)$. Thus, the unique boundary map on $\csx$ from Theorem~\ref{thm:uniquemap-koop} coincides with the canonical trace, and in particular $\interior (\partial_F \G)^g$ is empty for every non-trivial $g\in \G$.

Conversely, if $\G$ has a topologically free boundary, then Corollary \ref{cor:C-Reps} applies to $C^*_{\la_{\G}}(\G)$, implying the canonical trace is the unique boundary map,  which is now faithful. Hence $C^*_{\la_{\G}}(\G)$ is simple by Proposition~\ref{prop:ideal}.
\end{proof}

As mentioned in the introduction, the importance of noncommutative boundary maps was first noted by Kennedy, who proved the following special case of Theorem~\ref{thm:uniquemap-koop}. 
\begin{corollary}[{\cite[Theorem 3.4]{Ken15}}]
A group $\G$ is $C^*$-simple iff the canonical trace is the unique boundary map on $C^*_{\la_{\G}}(\G)$.
\end{corollary}

\begin{proof}
As we saw in the proof of Corollary \ref{classical}, if $\G$ is $C^*$-simple, then Theorem~\ref{thm:uniquemap-koop} implies that the canonical trace is the unique boundary map on $C^*_{\la_{\G}}(\G)$.

The converse is immediate from Proposition~\ref{prop:ideal}.
\end{proof}

The following criterion was key in several (non-)$C^*$-simplicity results, including Le Boudec's example of a non-$C^*$-simple group with unique trace property \cite{LeBou}.
\begin{corollary}[{\cite[Proposition 1.9]{BKKO}}]
Let $\G$ be a $C^*$-simple group. If $X$ is a $\G$-boundary such that the point stabilizer $\G_x$ is amenable for some $x\in X$, then $X$ is topologically free.
\end{corollary}

\begin{proof}
Since $\la_{\G/\G_x} = \la_{\G}$, we conclude that the canonical trace is the unique boundary map on $C^*_{\la_{\G/\G_x}}(\G)$ described in Theorem \ref{thm:uniquemap-koop},  which implies that $\interior X^g$ is empty for every non-trivial $g\in \G$.
\end{proof}

The above simple unified conceptual arguments (see also Corollary \ref{cor:Raum} below), indeed show that the notion of boundary maps is the fundamental tool in this context.


Now, turning our attention towards traces, using Theorem \ref{thm:uniquemap-koop}, we give a complete description of tracial structure of $C^*_{\pi}(\G)$ for germinal representations $\pi$.  In the case that $\pi$ is a Koopman representation, the implications in the next theorem (i)$\impliedby$(ii)$\iff$(iii)$\iff$(iv) follow immediately from \cite[Theorem 31]{Raum}. 

\begin{thm}\label{faithful-trace}
Let $X$ be a $\G$-boundary, let $N={\rm Ker}(\G\act X)$ be the kernel of the action, and let $(\pi,\rho)$ be a germinal representation of $(\G,X)$. Then the following are equivalent:
\begin{itemize}
\item[(i)]
$C^*_{\pi}(\G)$ admits a trace;
\item[(ii)]
the induced action $\displaystyle\frac{\G}{N} \act X$ is topologically free;
\item[(iii)]
if $\sigma\in \Rep(\G)$ and $\sigma\prec \pi$, then $\la_{\G/N}\prec\sigma$.
\item[(iv)]
$\la_{\G/N}\prec\pi$.
\end{itemize}

If the above equivalent conditions hold, then $\one_N$ extends to the unique trace on $C^*_{\pi}(\G)$.
\end{thm}

\begin{proof}
(i)$\!\implies\!$(ii):
Let $\tau$ be a trace on $C^*_{\pi}(\G)$. Then the map $\psi(a):=\tau(a)\one_{\partial_F\G}$ is a boundary map on $C^*_{\pi}(\G)$, which is unique by Theorem \ref{thm:uniquemap-koop}. Assume that $\displaystyle\frac{\G}{N} \act X$ is not topologically free and take $g\in\G$, $g\notin N$ such that $\interior X^g\neq\emptyset$. Then $\emptyset\neq \Delta_g \neq \partial_F\G$, which implies by Theorem \ref{thm:uniquemap-koop} that $\psi(\pi(g))$ is not constant, a contradiction.

(ii)$\!\implies\!$(iii):
For every $g\in N$, $\interior X^g= X$. If $\displaystyle\frac{\G}{N} \act X$ is topologically free, then for every $g\notin N$, $\interior X^g= \emptyset$. Thus, by Theorem \ref{thm:uniquemap-koop}, if $\psi$ is the boundary map on $C^*_{\pi}(\G)$, then its restriction $\psi|_{\pi(\G)}$ to $\pi(\G)$ is $\one_N\cdot \one_{\partial_F\G}$. Suppose $\sigma\in \Rep(\G)$ and $\sigma\prec \pi$. By the uniqueness of $\psi$, $C^*_{\sigma}(\G)$ also admits a unique boundary map $\psi'$, which restricts to $\one_N$ on $\pi(\G)$. Thus, composing $\psi'$ with a delta measure $\delta_z$, $z\in \partial_F\G$, we get a state on $C^*_{\sigma}(\G)$ whose restriction to $\pi(\G)$ is $\one_N$. Since $\delta_N\in \ell^2(\G/N)$ is a cyclic vector for $\la_{\G/N}$, we conclude $\la_{\G/N}\prec\sigma$.

(iii)$\!\implies\!$(iv): This is trivial. 

(iv)$\!\implies\!$(i): Note that $C^*_{\la_{\G/N}}(\G)$ admits a trace, which is the unique extension of $\one_N$. Thus, if $\la_{\G/N}\prec\pi$, then $C^*_{\pi}(\G)$ also admits a trace.
\end{proof}

One direction of the following corollary, namely the fact that topological freeness implies existence of trace, in the case of Koopman representations is known (see \cite{Raum}).

\begin{corollary}\label{cor:faithful-trace}
Let $X$ be a faithful $\G$-boundary, and $(\pi,\rho)$ a germinal representation of $(\G,X)$. Then $C^*_{\pi}(\G)$ admits a trace if and only if $X$ is topologically free. 

If $X$ is topologically free, then ${\pi}\succ\la_\G$ and the canonical trace is the unique trace on $C^*_{\pi}(\G)$.
\end{corollary}

\begin{remark}
Recently, the notion of topological boundaries has been generalized to the quantum setting in \cite{KKSV}, 
and applications in $C^*$-simplicity and the unique trace property 
of discrete quantum groups are investigated. 
A major obstacle in the theory of quantum group actions is the lack of an appropriate generalization of (topological) freeness. This, in part, prevents one to smoothly import the ideas and techniques of $C^*$-simplicity from groups to quantum groups. The above corollary, specially the forward implication, offers an interesting alternative to topological freeness for boundary actions, which may be the right notion to consider in the quantum case. In particular, the following conjecture seems natural: a discrete quantum group $\mathbb{\Gamma}$ is $C^*$-simple if the $C^*$-algebra of a quasi-regular or Koopman representation of a faithful boundary of $\mathbb{\Gamma}$ admits a trace.
\end{remark}

As a special case of Theorem \ref{faithful-trace} for $\pi$ the Koopman representations of quasi-invariant measures on $X$, we recover Raum's recent characterization of $C^*$-simplicity in \cite{Raum}.

\begin{corollary}\label{cor:Raum}\cite[Theorem 31]{Raum}
Let $X$ be a $\G$-boundary and $\k$ the Koopman representation associated to some quasi-invariant measure on $X$. Then the following statements are equivalent.

\begin{itemize}
\item[(i)]
$X$ is topologically free;
\item[(ii)]
if $\sigma\in \Rep(\G)$ and $\sigma\prec \k$, then $\la_{\G}\prec\sigma$;
\item[(iii)]
$\la_{\G}\prec\k$.
\end{itemize}
\end{corollary}

\begin{proof}
By Proposition \ref{prop:grpoi-reps-exam}, $\k$ is comes from a germinal representation of $(\G, X)$. Thus, the equivalence of the above statements follow from Theorem \ref{faithful-trace}.
\end{proof}

The following corollary of Theorem \ref{faithful-trace} is a connection between existence of trace and $C^*$-simplicity.
\begin{corollary}
Let $X$ be a $\G$-boundary, let $N={\rm Ker}(\G\act X)$ be the kernel of the action, and let $(\pi,\rho)$ be a germinal representation of $(\G,X)$. If $C^*_{\pi}(\G)$ admits a trace, then $C^*_{\pi}(\G)$ has a unique trace and a proper ideal $I_{\rm max}$ that contains all proper ideals of $C^*_{\pi}(\G)$ and $\displaystyle\frac{C^*_{\pi}(\G)}{I_{\rm max}} = C^*_{\la_{\sfrac{\G}{N}}}(\G)$ is simple.
\end{corollary}

\begin{proof}
If $C^*_{\pi}(\G)$ admits a trace, then $C^*_{\pi}(\G)$ has a unique trace by Theorem \ref{faithful-trace}. The existence of the ideal $I_{\rm max}$ with stated properties follows from parts (iii) and (iv) of Theorem \ref{faithful-trace}, and part (ii) in that theorem implies {\small$\displaystyle\frac{\G}{N}$} has a topologically free boundary action, hence $C^*_{\la_{\sfrac{\G}{N}}}(\G)$ is simple by Theorem \ref{Theo-KK}.
\end{proof}

\begin{remark}
In general, it is not true that, given $\pi\in\Rep(\G)$, if $C^*_\pi(\G)$ admits a unique boundary map, then $C^*_\pi(\G)$ is simple. For example, let $\G:=PSL_d(\Z)$ for some $d\geq 3$. Then the projective space  $\mathbb{P}(\mathbb{\R}^d)$ is a topologically free $\G$-boundary and there is $x\in\mathbb{P}(\mathbb{\R}^d)$ such that $\G_x$ is non-amenable. In this case, by Corollary \ref{cor:faithful-trace}, the unique boundary map on $\csx$ is the canonical trace. If $\csx$ were simple, then we would have that $\la_{\G/\G_x}\sim\la_\G$, which contradicts the fact that $\G_x$ is not amenable.
\end{remark}

\begin{remark}
In \cite[Proposition 3.1]{Ken15}, Kennedy showed that, given a group $\G$, there is a bijective correspondence between $\G$-boundaries in the state space $S(C^*_{\la_\G}(\G))$ and boundary maps on $C^*_{\la_\G}(\G)$. Actually, his proof applies word for word to any $\G$-$C^*$-algebra $A$. The correspondence takes a $\G$-boundary $X\subset S(A)$ into the boundary map $\psi\colon A\to C(\partial_F\G)$ given by $\psi(a)(y)=\b_{\hspace{-0.1em} X}(y)(a)$, for $a\in A$ and $y\in\partial_F\G$.
\end{remark}


Now going back to the description of the boundary map $\psi$ in Theorem \ref{thm:uniquemap-koop}, a natural question is, when is $\psi$ mapped into $C(X)$?

The following result gives a characterization in terms of Hausdorff germs.

\begin{proposition}\label{Haus}
Let $X$ be a $\G$-boundary and $(\pi,\rho)$ a germinal representation of $(\G,X)$. There exists a $\G$-map $\psi\colon C^*_\pi(\G)\to C(X)$ if and only if the action $\G\act X$ has Hausdorff germs.
Furthermore, in this case,
\begin{equation}\label{condint}
\psi(\pi(g))=\one_{\interior\! X^g}
\end{equation}
for every $g\in\G$.
\end{proposition}

\begin{proof}
If $\interior X^g$ is closed for each $g\in\G$, then clearly the image of the map $\psi$ from Theorem \ref{thm:uniquemap-koop} is contained in the copy of $C(X)$ inside $C(\partial_F\G)$, and, under this identification, $\psi$ satisfies \eqref{condint}.

Conversely, suppose that there exists a $\G$-map $\psi\colon C^*_\pi(\G)\to C(X)$. By Theorem \ref{thm:uniquemap-koop}, given $g\in\G$ there exists $K\subset X$ clopen such that $\b_{\hspace{-0.1em} X}^{-1}(K)=\Delta_g$. In particular, $\interior X^g\subset K\subset X^g$. Since $K$ is clopen, we conclude that $\interior X^g=K$ is clopen.
\end{proof}

\begin{proposition}\label{tc}
Let $X$ be a $\G$-boundary and $(\pi,\rho)$ a germinal representation of $(\G,X)$. Consider the following conditions:
\begin{enumerate}
\item[(i)] $C^*_{\pi\times\rho}(\G,X)$ is simple;
\item[(ii)] The unique boundary map on $C^*_{\pi\times\rho}(\G,X)$ is faithful;

\item[(iii)] $C^*_{\pi}(\G)$ is simple.
\end{enumerate}
Then (i)$\iff$(ii)$\implies$(iii).
\end{proposition}

\begin{proof}
(i)$\implies$(ii): Let $\psi$ be the unique boundary map on $C^*_{\pi\times\rho}(\G,X)$. By the Schwarz inequality, notice that $J_{\psi}:=\{a\in C^*_{\pi\times\rho}(\G,X):\psi(a^*a)=0\}$ is a $\G$-equivariant left ideal of $C^*_{\pi\times\rho}(\G,X)$. Given $a\in J_{\psi}$, $f\in C(X)$ and $g\in\G$, let us show that $a\rho(f)\pi(g)\in J_\psi$. Since $\rho(C(X))$ is contained in the multiplicative domain of $\psi$ by Theorem \ref{thm:uniquemap-koop}, we have that 

\begin{align*}
\psi((a\pi(g)\rho(f))^*(a\pi(g)\rho(f)) &= \psi(\rho(\overline{f})\pi(g^{-1})a^*a\pi(g)\rho(f)) \\
&= \psi(\rho(\overline{f}))\psi(g^{-1}(a^*a))\psi(\rho(f))\\
&=0.
\end{align*}

Since the linear span of elements of the form $\pi(g)f$ is dense in $C^*_{\pi\times\rho}(\G)$, we conclude the $J_\psi$ is an ideal, hence $J_\psi=0$.

(ii)$\implies$(i): This follows along the same lines of Proposition \ref{prop:ideal}.

(ii)$\implies$(iii): If the unique boundary map $\psi$ on $C^*_{\pi\times\rho}$ is faithful, then $\psi|_{C^*_\pi(\G)}$ is faithful as well, hence $C^*_\pi(\G)$ is simple by Proposition \ref{prop:ideal}.
\end{proof}

\section{$C^*$-simplicity of quasi-regular representations}\label{sec:simpl}

In this section, we turn our attention to the question of 
$C^*$-simplicity of quasi-regular representations.

\begin{lemma}\label{tecint}
Let $X$ be a $\G$-boundary. Given $x\in \Xoo$ and $y\in \partial_F\G$ such that $\b_{\hspace{-0.1em} X}(y)=x$, we have that $\G_x^0=\{h\in\G:y\in\Delta_h\}$.
\end{lemma}

\begin{proof}
Given $h\in \G_x^0$, we have that $x\in\interior X^h$, hence $y\in\b_{\hspace{-0.1em} X}^{-1}(\interior X^h)$. 

Conversely, given $h\in \G$ such that $y\in\Delta_h$, take a net $(z_i)\subset \b_{\hspace{-0.1em} X}^{-1}(\interior X^h)$ such that $z_i\to y$. In particular, $\G_{\b_{\hspace{-0.1em} X}(z_i)}^0\to\G_x^0$. Since $h\in\G_{\b_{\hspace{-0.1em} X}(z_i)}^0$ for every $i$, we conclude that $h\in\G_x^0$.
\end{proof}
\begin{theorem}\label{thm:prec}
Let $X$ be a $\G$-boundary and $(\pi,\rho)$ a germinal representation of $(\G,X)$. Given $x\in \Xoo$ and $\sigma\in\Rep(\G)$ such that $\sigma\prec\pi$, we have that $\la_{\G/\G_x^0}\prec\sigma$. 
\end{theorem}

\begin{proof}
Let $\psi$ be a boundary map on $ C^*_\sigma(\G)$. By Theorem \ref{thm:uniquemap-koop}, we have that $\psi(\sigma(g))=\one_{\Delta_g}$ for each $g\in \G$. 

Take $y\in\partial_F\G$ such that $\b_{\hspace{-0.1em} X}(y)=x$. It follows from Lemma \ref{tecint} that the composition of $\psi$ with delta measure on $y$ is a state on $C^*_\sigma(\G)$ which restricts to $\one_{\G_x^0}$ on $\G$. Since $\delta_{\G_x^0}$
 is a cyclic vector for $\csxo$, the result follows.
\end{proof}

The next result is an immediate consequence of Theorem \ref{thm:prec}.

\begin{corollary}\label{cor:xoo}
Let $X$ be a $\G$-boundary and $x\in \Xoo$. Then $\csxo$ is simple.
\end{corollary}

We give an example in the next section which shows that the assumption of $x\in \Xoo$ in the above Corollary is necessary and cannot be removed (see Example~\ref{counter-ex}).

\begin{remark}
In \cite[Corollary 8.5]{Kaw}, Kawabe showed that, given a group $\G$ and $x\in\partial_F\G$, then $\csx$ is simple. Since $\G_y=\G_y^0$ for each $y\in\partial_F\G$ (\cite[Lemma 3.4]{BKKO}), Kawabe's result also follows from Corollary \ref{cor:xoo}.
\end{remark}

We now prepare for the main result of this section (Theorem~\ref{thm:C*-simple-main}), which gives a complete characterization of when $\csx$ is simple. For that we need the following few results.

\begin{lemma}\label{lem:diamond}
Let $\pi,\sigma\in\Rep(\G)$ and suppose that $C^*_\pi(\G)$ and $C^*_\sigma(\G)$ admit boundary maps $\psi_\pi$ and $\psi_\sigma$ such that $\psi_\pi\circ\pi=\psi_\sigma\circ\sigma$. If $C^*_\sigma(\G)$ is simple, then $\sigma\prec\pi$.
\end{lemma}
\begin{proof}
Consider the following commutative diagram:

\centerline{\xymatrix{
&C^*(\G)\ar[dl]_{\pi}\ar[dr]^{\sigma}\\
C^*_\pi(\G)\ar[dr]_{\psi_\pi}&&
C^*_\sigma(\G)\ar[dl]^{\psi_\sigma}\\
&C(\partial_F\G)}}

Since $C^*_\sigma(\G)$ is simple, we conclude from Proposition \ref{prop:ideal} that $\psi_\sigma$ is faithful. Therefore, $\ker\pi\subset\ker\sigma$ and $\sigma\prec\pi$.
\end{proof}

\begin{proposition}\label{prop:subg}
Let $X$ be a $\G$-boundary. Given $x\in X$ and $\G_x^0\leq\La\leq\G_x$ such that $C^*_{\la_{\G/\La}}(\G)$ is simple, we have that {\small$\displaystyle\frac{\La}{\G_x^0}$} is amenable.
\end{proposition}
\begin{proof}
Since $\G_x^0\leq\La\leq\G_x$,  it follows from Proposition \ref{prop:grpoi-reps-exam} that $\la_{\G/\La}$ is a germinal representation.  From Theorem \ref{thm:uniquemap-koop} and Lemma \ref{lem:diamond}, we conclude that $\la_{\G/\La}\prec\la_{\G/\G_x^0}$.  Therefore, $\frac{\La}{\G_x^0}$ is amenable by Proposition \ref{prop:co-am}.
\end{proof}

\begin{proposition}\label{prop:intermediate}
Let $X$ be a $\G$-boundary. Given $x\in X$, there exists $\La$ such that $\G_x^0\leq\La\leq\G_x$ and $C^*_{\la_{\G/\La}}(\G)$ is simple.
\end{proposition}
\begin{proof}
Take $y\in\partial_F\G$ such that $\b_X(y)=x$, and let $\La:=\{h\in\G:y\in\Delta_g\}$. It is straightforward to check that $\G_x^0\subset \La\subset \G_x$. 

Let $(\pi,\rho)$ be a germinal representation of $(\G,X)$ and $\psi$ be the unique boundary map on $C^*_{\pi}(\G)$ as in Theorem~\ref{thm:uniquemap-koop}. Then the positive-definite function $\delta_y\circ\psi|_\G$ coincides with $\one_\La$. In particular, $\{\pi(g):g\in\La\}$ is in the multiplicative domain of $\delta_y\circ\psi|_\G$ and $\La$ is a subgroup. 

Arguing as in Theorem \ref{thm:prec}, we conclude that any quotient of $C^*_{\la_{\G/\La}}(\G)$ weakly contains $\la_{\G/\La}$ and $C^*_{\la_{\G/\La}}(\G)$ is simple.
\end{proof}

\begin{thm}\label{thm:C*-simple-main}
Let $X$ be a $\G$-boundary, and $x\in X$. Then $\csx$ is simple iff the quotient group {\small$\displaystyle \frac{\G_x}{\G_x^0}$} is amenable. 
\end{thm}

\begin{proof}
The forward implication follows from Proposition \ref{prop:subg}.

Conversely, assume {\small$\displaystyle \frac{\G_x}{\G_x^0}$} is amenable and take $\La$ as in Proposition \ref{prop:intermediate}. Since {\small$\displaystyle \frac{\G_x}{\G_x^0}$} is amenable, we have that $\La$ is co-amenable in $\G_x$, and so $\la_{\G/\G_x}\prec\la_{\G/\La}$ by Proposition~\ref{prop:co-am}. Hence, $\csx$ is simple.
\end{proof}

We can improve the above result in the case of points $x\in \Xoo$.

\begin{thm}\label{closed simple}
Let $X$ be a $\G$-boundary. Given $x\in \Xoo$ and $\G_x^0\leq\La\leq\G_x$, we have that $C^*_{\la_{\sfrac{\G}{\La}}}\!(\G)$ is simple iff the quotient group {\small$\displaystyle\frac{\La}{\G_x^0}$} is amenable.
\end{thm}

\begin{proof}
The forward implication follows from Proposition \ref{prop:subg}.

Conversely, if {\small$\displaystyle\frac{\La}{\G_x^0}$} is amenable, then $\la_{\La/\G_x}\prec \la_{\G/\G_x^0}$ by Proposition \ref{prop:co-am}, hence $C^*_{\la_{\sfrac{\G}{\La}}}\!(\G)$ is simple by Corollary~\ref{cor:xoo}.
\end{proof}

We end the section with considering the simplicity problem for the $C^*$-algebras generated by the germinal representations $(\pi,\rho)$.
The following result is the version of Theorem \ref{thm:prec} in this setup. 

\begin{theorem}\label{thm:grp}
Let $X$ be a $\G$-boundary. Given a germinal representation $(\pi,\rho)$ of $(\G,X)$ and $x\in \Xoo$, we have that $(\la_{\G/\G_x^0},\cP_x)\prec (\pi,\rho)$. 
\end{theorem}
\begin{proof}
Let $\psi$ be the unique boundary map on $C^*_{\pi\times\rho}(\G,X)$. By Theorem \ref{thm:uniquemap-koop}, we have that $\psi(\pi(g))=\one_{\Delta_g}$ and $\psi(\rho(f))=f\circ\b_{\hspace{-0.1em} X}$ for every  $g\in\G$ and $f\in C(X)$.

Choose $y\in\partial_F\G$ such that $\b_{\hspace{-0.1em} X}(y)=x$ and let $\tau_y$ be the state on $C^*_{\pi\times\rho}(\G,X)$ given by composing $\psi$ with the delta measure on $y$. By Lemma \ref{tecint}, 
$$\tau_y(\pi(g)\rho(f))=\one_{\G_x^0}(g)f(x)=\langle\la_{\G/\G_x^0}(g)\cP_x(f)\delta_{\G_x^0},\delta_{\G_x^0}\rangle,$$ 
for $g\in\G$ and $f\in C(X)$. Since the vector $\delta_{\G_x^0}$ is cyclic for $C^*_{\la_{\G/\G_x^0}\times\cP_x}(\G,X)$, we conclude that $(\la_{\G/\G_x^0},\cP_x)\prec (\pi,\rho)$.
\end{proof}

The following version of Corollary \ref{cor:xoo} is immediate from Theorem \ref{thm:grp}.

\begin{corollary}\label{cor:cp}
Let $X$ be a $\G$-boundary and $x\in \Xoo$. Then $C^*_{\la_{\G/\G_x^0}\times\cP_x}(\G,X)$ is simple.
\end{corollary}

\section{Quasi-regular representations of Thompson's groups}\label{sec:thomp}

In this section, we apply the results of the previous section to analyze certain quasi-regular representations of \emph{Thompson's groups}.

Recall that Thompson's group $V$ is the set of piecewise linear bijections on $[0,1)$ which are right continuous, have finitely many points of non-differentiability, all being dyadic rationals, and have a derivative which is a power of $2$ at each point of differentiability. Thompson's group $T$ consists of those elements of $V$ which have at most one point of discontinuity, and Thompson's group $F$ consists of those elements of $V$ which are homeomorphisms of $[0,1)$ or, equivalently, the set of elements $g\in T$ satisfying $g(0)=0$.

\subsection{Boundary actions of Thompson's groups}

We will now recall the boundary action of Thompson's group $V$ on the Cantor set considered, for example, in \cite[Section 4.1]{LeM}. We thank Christian Skau for suggesting the description, similar to a construction from \cite{PSS}, that we present here of this action.

 Let $Y$ be the set obtained from $\R$ by replacing each $y\in\Z[1/2]$ by two elements $\{y_-,y_+\}$, and endow $Y$ with the order topology. Let $K:=Y\cap[0_+,1_-]$. Then $K$ is a Cantor set. 
 Let $\alpha$ be the action of $V$ on $K$ given by
\begin{align}
\alpha_g(s)&:=g(s)\nonumber\\
\alpha_g(y_+)&:=g(y)_+\nonumber\\
\alpha_g(y_-)&:=\left(\lim_{s\underset{\ell}{\rightarrow} y}g(s)\right)_-\label{limit}
\end{align}
for $g\in V$, $s\in(0,1)\setminus\Z[1/2]$ and $y_\pm\in K$, where in the limit \eqref{limit}, the notation $s\underset{\ell}{\rightarrow} y$ denotes convergence of $s$ to $y$ from the left.

Given $I,J\subset [0,1)$ left-closed and right-open intervals with endpoints in $\Z[1/2]$, there exist a piecewise linear homeomorphism $f\colon I\to J$ with a derivative which is a power of $2$ at each point of differentiability and with finitely many points of non-differentiability, all of which belong to $\Z[1/2]$ (see \cite[Lemma 4.2]{CFP}). It follows easily from this fact that $T\act K$ is an extreme boundary action.

\begin{theorem}\label{applFT}
The $C^*$-algebra $C^*_{\la_{T/F}}(T)$ admits no traces and is simple.
\end{theorem}
\begin{proof}
 Consider $T\act K$ and observe that $T_{0_+}=F$. Clearly, $K$ is a faithful but not topologically free $T$-boundary. From Theorem \ref{faithful-trace}, we conclude that $C^*_{\la_{T/F}}(T)$ does not admit traces.

Notice that $T_{0_+}^0=\{g\in F:g'(0)=1\}$. Therefore, the map $F/T_{0_+}^0\to\Z$ which sends $gT_{0_+}^0\in F/T_{0_+}^0$ such that $g'(0)=2^a$ to $a$, is an isomorphism.

Since $\frac{T_{0_+}}{T^0_{0_+}}\simeq \Z$ is amenable, we conclude from Theorem \ref{thm:C*-simple-main} that $C^*_{\la_{T/F}}(T)$ is simple.
\end{proof}

Notice that the set $\cX:=[0,1)\cap\Z[1/2]$ is invariant under the action of $V$ on $[0,1)$. Denote by $\pi$ the associated unitary representation of $V$ on $\ell^2(\cX)$. Clearly, the action of $T$ on $\cX$ is transitive, and $F$ is the stabilizer of $0$. Therefore, $\pi|_T\sim_u\la_{T/F}$. Denote by $H$ the stabilizer of $0$ with the respect to the $V$-action. By the same reason, $\pi\sim_u\la_{V/H}$. The same arguments of Theorem \ref{applFT} can be used for showing that $C^*_\pi(V)$ is simple and admits no traces.

Recall that a unital $C^*$-algebra $A$ is said to be \emph{stably finite} if, for every $n\in \N$, $M_n(A)$ does not contain any infinite projection. If $A$ is simple, then $A$ is said to be \emph{purely infinite} if every non-zero hereditary $C^*$-subalgebra of $A$ contains an infinite projection.

In \cite[Proposition 4.3]{HaagOle}, Haagerup and Olesen showed that $C^*_\pi(F)\subsetneq C^*_\pi(T)\subsetneq C^*_\pi(V)=\mathcal{O}_2$, where $\cO_2$ is the Cuntz algebra. In particular, $C^*_\pi(V)$ is a \emph{Kirchberg algebra} (i.e., a separable, simple, nuclear and purely infinite $C^*$-algebra). 

The fact that $C^*_{\la_{T/F}}(T)\subset\cO_2$ implies that $C^*_{\la_{T/F}}(T)$ is an exact $C^*$-algebra. Furthermore, since $C^*_{\la_{T/F}}(T)$ admits no traces, a deep result of Haagerup (\cite[Corollary 5.12]{Haa14}) implies that $C^*_{\la_{T/F}}(T)$ is not stably finite.

\begin{question}
Is $C^*_{\la_{T/F}}(T)$ a Kirchberg algebra?
\end{question}

We conclude this section with the example which we promised in the last section to show that the assumption in Corollary~\ref{cor:xoo} is necessary.

First, we recall the following Proposition, which is known (\cite{HaagOle, Oza-note}), and was behind the observation by Haagerup and Olesen that $\la_T\nprec\la_{T/F}$ (see \cite{HaagOle}).

Given a bijection $g$ on a set $\Omega$, let $\supp g:=\{x\in \Omega:gx\neq x\}$. Notice that $g(\supp g)=\supp g$.

\begin{proposition}\label{prop:HaagOle}
 Let $\G$ be a group acting on a set $\Omega$ and $\la\colon\G\to B(\ell^2(\Omega))$ be the associated unitary representation. Given $g,h\in \G$, we have that $\supp g\cap\supp h=\emptyset$ iff $\la_{gh}=\la_g+\la_h-1$.
\end{proposition}
\begin{proof}

Suppose that there exists $x\in\supp g\cap \supp h$. Then $\langle\la_{gh}(\delta_x),\delta_x\rangle\geq0$, whereas $\langle(\la_g+\la_h-1)(\delta_x),\delta_x\rangle=-1.$ Hence, $\la_{gh}\neq \la_g+\la_h-1$.

Suppose that $\supp g\cap \supp h=\emptyset$. If $y\notin\supp g\cup\supp h$, then clearly $\delta_y=\la_{gh}(\delta_y)=(\la_g+\la_h-1)(\delta_y)$.

If $y\in\supp_g$, then $\delta_{gy}=\la_{gh}(\delta_y)=(\la_g+\la_h-1)(\delta_y)$. Finally, if $y\in\supp h$, then $hy\in\supp h$ as well, and $\delta_{hy}=\la_{gh}(\delta_y)=(\la_g+\la_h-1)(\delta_y)$.
\end{proof}

\begin{example}\label{counter-ex}
Identifying $S^1$ with $[0,1)$, we get a boundary action of $T$ and the open stabilizer of $0$ coincides with the commutator subgroup $[F,F]$ (\cite[Theorem 4.1]{CFP}). 

Let $\Omega:=\Z[1/2]\cap[0,1)$. Given $x\in \Omega$, let $D_+(g)\colon \Omega\to\Z$ be given by $D_+(g)(x)=n$ if $\lim_{y\to x^+}g'(y)=2^n$. Define $D_-$ analogously, so that $D_+$ and $D_-$ are the maps obtained by taking right and left derivatives, respectively, followed by $\log_2$. 
It holds that $D_\pm(gh)(x)=D_\pm(g)(h(x))+D_\pm(h)(x)$, for any $g,h\in T$ and $x\in \Omega$.

Consider the action of $T$ on $\Omega\times\Z\times\Z$ given by $$g(x,n,m)=(g(x),n+D_-(g)(x),m+D_+(g)(x)).$$
It is easy to see that this action is transitive and the stabilizer of $(0,0,0)$ coincides with $[F,F]$. Therefore, $T/[F,F]$ can be identified with $\Omega\times\Z\times\Z$.

Choose $g\in T$ such that $g(x)=x$ for $x\geq 1/2$, and $D_-(g)(1/2)\neq 0$. Also choose $h\in T $ such that $h(x)=x$ for $x\leq 1/2$ and $D_+(h)(1/2)\neq 0$. Then the supports of $g$ and $h$ with respect to the action of $T$ on $\Omega$ are disjoint, but the point $(1/2,0,0)$ is in the intersection of the supports of $g$ and $h$ with respect to the action on $\Omega\times\Z\times\Z$. Using Proposition~\ref{prop:HaagOle}, 
$$(1-\la_{T/[F,F]}(g))(1-\la_{T/[F,F]}(h))\neq 0=(1-\la_{T/F}(g))(1-\la_{T/F}(h)).$$
Therefore, $\la_{T/[F,F]}\nprec\la_{T/F}$. On the other hand, it follows from Proposition~\ref{prop:co-am} that $\la_{T/F}\prec\la_{T/[F,F]}$. 
This show that the $C^*$-algebra $C^*_{\la_{T/[F,F]}}(T)$ is not simple, and therefore the hypothesis in Corollary~\ref{cor:xoo} that $x\in \Xoo$ is necessary.
\end{example}

\section{Groupoids of germs}\label{section:groupoid}

In this section, we present a characterization of germinal representations of compact $\G$-spaces, as introduced in Definition \ref{grprep}.

We start by recalling some facts about étale groupoids and groupoids of germs. The reader can find more details in e.g. \cite[Chapter 3]{Put} for the Hausdorff case, and \cite{Exel} for the general case.

\subsection{Étale groupoids}

An \emph{étale groupoid} $ G $ is a topological groupoid whose unit space $ G ^{(0)}$ is locally compact and Hausdorff, and such that the range and source maps $r,s\colon G\to G^{(0)}$ are local homeomorphisms. Given $x\in G^{(0)}$, the \emph{isotropy group} of $G$ at $x$ is the group $r^{-1}(x)\cap s^{-1}(x)$. If $\interior\{g\in G:s(g)=r(g)\}=G^{(0)}$, then $G$ is said to be \emph{effective}.

An open subset $U\subset G $ is said to be a \emph{bisection} if the restrictions of $r$ and $s$ to $U$ are injective. Notice that every bisection is a locally compact Hausdorff space.

Given a bisection $U$, we view a function in $C_c(U)$, as a function on $ G $ by defining it to be $0$ outside of $U$.

Let $\cC( G )$ be the linear span within the space of all complex-valued functions on $ G $ of the union of the $C_c(U)$ for all open bisections $U$. If $ G $ is not Hausdorff, then the functions in $\cC( G )$ are not necessarily continuous. 

We will need the following result later on:

\begin{proposition}[{\cite[Proposition 3.10]{Exel}}]\label{exel}
Let $\cU$ be a collection of bisections of $G$ such that $G=\bigcup \cU$. Then $\cC(G)$ is linearly spanned by the collection of all subspaces of the form $C_c(U)$, where $U\in\cU$.
\end{proposition}

The vector space $\cC( G )$ has the structure of a $*$-algebra with product given by $(f_1 f_2)(g):=\sum_{g_1g_2=g}f_1(g_1)f_2(g_2)$ and involution by $f_1^*(g)=\overline{f_1(g^{-1})}$, for $f_1,f_2\in \cC( G )$ and $g\in G $.

Given $x\in G^{(0)}$, there is a $*$-homomorphism $L_x\colon\cC(G)\to B(\ell^2(s^{-1}(x)))$ given by 
\begin{equation}\label{reducedr}
L_x(a)\delta_h=\sum_{g\in s^{-1}(x)} a(gh^{-1})\delta_g,
\end{equation}
 for $a\in\cC(G)$ and $h\in s^{-1}(x)$. The \emph{reduced $C^*$-algebra} of $G$, denoted by $C^*_r(G)$, is the completion of $\cC(G)$ under the $C^*$-norm given by $\|a\|:=\sup_{x\in G^{(0)}}\|L_x(a)\|$.

The \emph{full $C^*$-algebra} of $ G $, denoted by $C^*( G )$, is the completion of $\cC( G )$ under the $C^*$-norm given by $|||a|||:=\sup\{\|\sigma(a)\|:\text{$\|\cdot\|$ is a $C^*$-seminorm on $\cC(G)$}\},$ for $a\in \cC( G )$.

From now on, we assume that $ G ^{(0)}$ is compact. In this case, $\cC( G )$ is unital and the inclusion $C(G^{(0)})\to\cC(G)$ is a unital $*$-homomorphism. A \emph{representation} of $\cC( G )$ on a Hilbert space $\cH$ is a unital $*$-homomorphism $\sigma\colon \cC( G )\to B(\cH)$. As in the case of groups, given representations $\sigma_1,\sigma_2$ of $G$, we say that $\sigma_1$ weakly contains $\sigma_2$, and denote this by $\sigma_2\prec\sigma_1$, if $\|\sigma_2(f)\|\leq\|\sigma_1(f)\|$ for every $f\in \cC(G)$.

 If $G$ is Hausdorff, then the restriction map $E\colon\cC(G)\to C(G^{(0)})$ extends to a conditional expectation $E\colon C^*_r(G)\to C(G^{(0)})$.

\subsection{Groupoids of germs}

Let $X$ be a compact $\G$-space. Given $(g,x)$ and $(h,y)$ in $\G\times X$ we will say that $(g,x)\sim(h,y)$ if $x=y$ and there is a neighborhood $U$ of $x$ such that $g|_U=h|_U$. The equivalence class of $(g,x)$ will be denoted by $[g,x]$. We say that $ G(\G,X): =\frac{\G\times X}{\sim}$ is the \emph{groupoid of germs} of the action.

The product of two elements $[h,y]$ and $[g,x]$ is defined if and only if $y=gx$, in which case $[h,y][g,x]=[hg,x]$. Inversion is given by $[g,x]^{-1}:=[g^{-1},gx]$. We identify $G^{(0)}$ with $X$.

As observed in \cite{JNS}, given $x\in X$, the isotropy group of $G(\G,X)$ at $x$ is naturally identified with $\G_x/\G_x^0$.

Given $U\subset X$ open and $g\in \G$, let $\Theta( g,U):=\{[g,x]:x\in U\}$. The topology on $ G(\G,X)$ is the one generated by the basis $\{\Theta( g,U):g\in \G, \text{$U\subset X$ open}\}$. With this topology, $ G(\G,X) $ is an effective étale groupoid and each set of the form $\Theta( g,U)$ is a bisection.

As observed in e.g. \cite[Lemma 4.3]{NyO}, $ G(\G,X)$ is Hausdorff if and only if $\interior X^g$ is closed in $X$ for each $g\in\G$.

\subsection{Germinal representations}

In the following, given a compact $\G$-space $X$, we show that there is a bijection between representations of $\cC(G(\G,X))$ and germinal representations of $(\G,X)$. We follow the ideas of Exel (\cite{Exel}), who presented a similar characterization in the more general case of inverse semigroup actions. 

Fix $X$ a compact $\G$-space. Given $g\in\G$, we denote the compact bisection $\Theta(g,X)$ by $\Theta_g$, and let $\delta_g:=\one_{\Theta_g}\in \cC(G(\G,X))$. 

\begin{proposition}\label{many}
Given $g,h\in\G$, we have that
\begin{enumerate}
\item[(i)] $\delta_g\delta_h=\delta_{gh}$;
\item[(ii)] $(\delta_g)^*=\delta_{g^{-1}}$;
\item[(iii)]  $\delta_g f=(gf)\delta_g,$ for every $f\in C(X)$;

\item[(iv)] The map $f\in C(X)\to\delta_gf\in C(\Theta_g)$ is a linear isomorphism;
\item[(v)] $\Theta_g\cap\Theta_h=\Theta(g,\interior X^{g^{-1}h})$; 
\item[(vi)] $\delta_g f = f,$ for every $f\in C_c(\interior X^g)$.

\end{enumerate}
\end{proposition}
\begin{proof}
The proofs of (i)-(iv) are straightforward computations and are a special case of \cite[Propositions 7.3 and 7.5]{Exel}, so we omit them. 

(v) This follows from the fact that, given $x\in X$, we have that $[g,x]\in\Theta_h$ if and only if $x\in\interior X^{g^{-1}h}$.

(vi) Given $x\in X$, we have that $x\in\interior X^g$ if and only if $[g,x]=[e,x]$. Hence, if $x\in\interior X^g$, then $[g,x]=[e,x]$ and $(\delta_gf)([g,x])=f([e,x])$. If $x\notin\interior X^g$, then $(\delta_gf)([g,x])=f([e,x])=0$, since $f\in C_c(\interior X^g)$. This concludes the proof of (vi).
\end{proof}

\begin{lemma}\label{technical}
Let $X$ be a compact $\G$-space and $(\pi,\rho)$ a germinal representation of $(\G,X)$. Let $J$ be a finite subset of $\G$ and suppose that for each $g\in J$ we are given $f_g\in C(X)$ such that $\sum_{g\in J}\delta_gf_g=0$. Then $\sum_{g\in J}\pi(g)\rho(f_g)=0$.
\end{lemma}

\begin{proof}
Fix $\xi,\eta\in H_\pi$. For each $g\in J$, let $\mu_g$ be the finite regular measure on $\Theta_g$ such that, for each $f\in C(X)$, $\int\delta_gfd\mu_g=\langle\pi(g)\rho(f)\xi,\eta\rangle$. 

We claim that, given $g,h\in J$, $\mu_g$ coincides with $\mu_h$ on $\Theta_g\cap\Theta_h=\Theta(g,\interior X^{g^{-1}h})$. Indeed, given $f\in C_c(\interior X^{g^{-1}h})$, we have that $\pi(g^{-1}h)\rho(f)=\rho(f)$. Hence,
$$\int \delta_g fd\mu_g=\langle\pi(g)\rho(f)\xi,\eta\rangle=\langle\pi(h)\rho(f)\xi,\eta\rangle=\int\delta_hfd\mu_h.$$
This concludes the proof that $\mu_g$ and $\mu_h$ coincide on $\Theta_g\cap\Theta_h$.

Let $M$ be the open subset of $G(\G,X)$ given by $M:=\bigcup_{g\in J}\Theta_g$. Clearly, there is a finite Borel measure $\mu$ on $M$ such that $\mu(A)=\mu_g(A)$ for any $g\in J$ and $A\subset \Theta_g$ measurable. Then,
$$\left\langle\sum_{g\in J}\pi(g)\rho(f)\xi,\eta\right\rangle=\sum_{g\in J}\int\delta_gf_gd\mu_g=\int\sum_{g\in J}\delta_gf_g d\mu=0.$$
Since $\xi$ and $ \eta$ were arbitrary, we conclude that $\sum_{g\in J}\pi(g)\rho(f_g)=0$.
\end{proof}

Propositions \ref{exel}, \ref{many} and Lemma \ref{technical} imply that, given a compact $\G$-space $X$ and $(\pi,\rho)$ a germinal representation of $(\G,X)$, there exists a unique representation $\pi\times\rho$ of $\cC(G(\G,X))$ on $\cH_\pi$ given by $(\pi\times\rho)(\delta_gf)=\pi(g)f$, for every $g\in \G$ and $f\in C(X)$. Conversely, it follows easily from Proposition \ref{many} that any representation of $\cC(G(\G,X))$ is of this form. We summarize this discussion in the following:

\begin{theorem}\label{corr}
Let $X$ be a compact $\G$-space. There is a bijection between the class of germinal representations of $(\G,X)$ and the class of representations of $\cC(G(\G,X))$ given by the correspondence $(\pi,\rho)\to\pi\times\rho$. 
\end{theorem}

\begin{remark}\label{remark}
Let $X$ be a compact $\G$-space. 
\begin{enumerate}
\item[(i)]
Let $I$ be the ideal of $C(X)\rtimes\G$ generated by the relations in \eqref{eq:cond}. Theorem \ref{corr} implies that $C^*(G(\G,X))\simeq\frac{C(X)\rtimes\G}{I}$.

\item[(ii)]

Given $x\in X$, notice that $s^{-1}(x)=\{[g,x]:g\in\G\}$ can be identified with $\G/\G_x^0$. Under this identification, the representation $L_x$ from \eqref{reducedr} is equal to $\la_{\G/\G_x^0}\times \cP_x$.

\end{enumerate}
\end{remark}

\begin{theorem}\label{exotic}
Let $X$ be a $\G$-boundary with Hausdorff germs and $(\pi,\rho)$ a germinal representation of $(\G,X)$. Then there exists a commutative diagram of canonical surjective $*$-homomorphisms:

\centerline{\xymatrix{
C^*(G(\G,X))\ar[rr]\ar[dr]&&
C^*_r(G(\G,X))\\
&C^*_{\pi\times\rho}(\G,X)\ar[ur]}}
\end{theorem}
\begin{proof}

It follows from Theorem \ref{thm:grp} and Remark \ref{remark}.(ii) that $$\bigoplus_{x\in X}L_x\prec\pi\times\rho.$$ 
Since $C^*_r(G(\G,X))$ is the completion of $\cC(G(\G,X))$ under the representation $$\bigoplus_{x\in X}L_x,$$ the result follows.
\end{proof}
\begin{remark}
Let $X$ be a $\G$-boundary with Hausdorff germs.
\begin{enumerate}
\item[(i)]
In \cite{KS}, Kyed and So\l tan introduced the term \emph{exotic} for a completion of the polynomial algebra on a discrete quantum group which sits in between the maximal and minimal completions (later, this notion was investigated in other contexts by several people). Hence, another way to phrase Theorem \ref{exotic} is to say that any germinal representation of $(\G,X)$ is ``exotic''.
\item[(ii)]
It follows from Proposition \ref{many} that the canonical conditional expectation $E\colon C^*_r(G(\G,X))\to C(X)$ is a $\G$-map. Hence, the unique $\G$-map $\psi\colon C^*_\pi(\G)\to C(X)$ from Proposition \ref{Haus} is the composition of the canonical map $C^*_\pi(\G)\to C^*_r(G(\G,X))$ with $E$.
\end{enumerate}
\end{remark}

\begin{example}
Consider the boundary action of Thompson's group $T$ on $S^1$ given by identifying the extremes of $[0,1]$. Since this action is faithful, but not topologically free, it follows from connectedness of $S^1$ that $G(T,S^1)$ is not Hausdorff. 

It follows from Example \ref{counter-ex} and Proposition \ref{tc} that $C^*_{\la_{T/T_1^0}\times\cP_1}(T,S^1)$ is not simple. Considering the identification from Remark \ref{remark}.(ii), we conclude that $C^*_r(G(T,S^1))$ is not simple, even though $G(T,S^1)$ is minimal and effective.

On the other hand, it follows from \cite[Remark 7.27]{KwaMey} that the \emph{essential $C^*$-algebra} of $G(T,S^1))$ is purely infinite and simple.
\end{example}

\begin{remark}

Given a Hausdorff étale groupoid $G$,  Kennedy, Kim, Li, Raum and Ursu have completely characterized simplicity of $C^*_r(G)$ in \cite{KKLRU} (more generally, in the not necessarily Hausdorff case their characterization applies to the essential $C^*$-algebra of $G$). 

 For non-Hausdorff groupoids, the question of simplicity of $C^*_r(G)$ was investigated in \cite{CEPSS}, but an intrinsic criterion ensuring simplicity is still missing. 
\end{remark}

\end{document}